\documentclass[12pt]{amsart}

\usepackage{amsfonts,amsthm,amsmath,amssymb,amscd,mathrsfs}
\usepackage{graphics}
\usepackage{indentfirst}
\usepackage{cite}
\usepackage{latexsym}
\usepackage[dvips]{epsfig}
\usepackage{color}
\newtheorem{theorem}{Theorem}[section]
\newtheorem{remark}{Remark}[section]

\newtheorem{lemma}[theorem]{Lemma}
\newtheorem{pro}[theorem]{Proposition}

\renewcommand{\div}{{\rm div }}

\newcommand{\bt}{\begin{theorem}}
\newcommand{\bl}{\begin{lemma}}
\newcommand{\el}{\end{lemma}}
\newcommand{\et}{\end{theorem}}

\newcommand{\curl}{{\rm curl} }

\newcommand{\la}{\label}

\newcommand{\bn}{\begin{eqnarray}}
\newcommand{\en}{\end{eqnarray}}
\newcommand{\bnn}{\begin{eqnarray*}}
\newcommand{\enn}{\end{eqnarray*}}
\newcommand{\ba}{\begin{aligned}}
\newcommand{\ea}{\end{aligned}}
\newcommand{\be}{\begin{equation}}
\newcommand{\ee}{\end{equation}}

\renewcommand{\la}{\label}

\newcommand{\Bv}{{\boldsymbol{v}}}
\newcommand{\Bw}{{\boldsymbol{w}}}

\newcommand{\Bn}{{\boldsymbol{n}}}
\newcommand{\Bt}{{\boldsymbol{\tau}}}
\newcommand{\Bu}{{\boldsymbol{u}}}
\newcommand{\Be}{{\boldsymbol{e}}}
\newcommand{\BF}{{\boldsymbol{F}}}

\newcommand{\bBU}{\bar{{\boldsymbol{U}}}}

\newcommand{\Bo}{{\boldsymbol{\omega}}}

\newcommand{\mcE}{\mathcal{E}}

\newcommand{\what}{\widehat}
\newcommand{\veps}{\varepsilon}

\def\XXint#1#2#3{{\setbox0=\hbox{$#1{#2#3}{\int}$ }
\vcenter{\hbox{$#2#3$ }}\kern-.6\wd0}}

\begin{document}

\title[Uniqueness and Stability of Poiseuille Flows]
{Uniqueness and uniform structural stability of Poiseuille flows in an infinitely long pipe with Navier boundary conditions}

\author{Yun Wang}
\address{School of Mathematical Sciences, Center for dynamical systems and differential equations, Soochow University, Suzhou, China}
\email{ywang3@suda.edu.cn}

\author{Chunjing Xie}
\address{School of mathematical Sciences, Institute of Natural Sciences,
Ministry of Education Key Laboratory of Scientific and Engineering Computing,
IMA-Shanghai, Shanghai Jiao Tong University, 800 Dongchuan Road, Shanghai, China}
\email{cjxie@sjtu.edu.cn}

\begin{abstract}
{In this paper, uniqueness and uniform structural stability of Poiseuille flows in an infinitely long pipe with Navier boundary conditions are established for axisymmetric solutions of steady Navier-Stokes system. The crucial point is that the estimate is uniform with respect both the flux of flows and slip coefficient which appeared in Navier boundary conditions. With the aid of special structure of Navier-Stokes system and the refined estimate for some quantities such as radial velocity, the uniqueness and existence of steady solutions of Navier-Stokes system can  be obtained even when the external forces are large as long as the fluxes of flows are large. The delicate decomposition in the two dimensional plane for  slip coefficient and frequency corresponding to Fourier variable in the axial direction plays a key role to achieve   these estimates.}
\end{abstract}


\date{}

\maketitle

\section{Introduction and Main Results}
An interesting problem in fluid mechanics is to study the flows in nozzles. Given an infinitely long nozzle $\Omega$, consider the following steady incompressible Navier-Stokes system
\begin{equation}\label{NS}
\left\{
\begin{aligned}
& \Bu\cdot \nabla \Bu - \Delta \Bu +\nabla p= \BF\ \ \ \ \mbox{in}\ \Omega,\\
&\div~\Bu=0\ \ \ \ \ \mbox{in}\ \Omega,
\end{aligned}
\right.
\end{equation}
where $\Bu=(u^x, u^y, u^z)$ and $\BF=(F^x, F^y, F^z)$ are the velocity field and external force, respectively. Here we have formally put the viscosity coefficient to be the unity.
When the nozzle is a straight cylinder, the system \eqref{NS} supplemented no slip boundary condition admits the shear flow solutions which are of the form  $\Bu=(u^x, u^y, u^z)=(0,0, u^z(x,y))$ and are called Poiseuille flows. When the infinitely long nozzle $\Omega$ tends to straight cylinders at far fields,
it was proposed by Leray in \cite{Leray} to study the well-posedness of the system \eqref{NS} supplemented with no slip boundary condition so that the solutions tend to the Poiseuille flows at far fields. This problem is called Leray problem nowadays. The first significant contribution to the solvability of Leray problem is due to Amick (\hspace{1sp}\cite{Amick1, Amick2}). He reduced the proof of existence to the resolution of a well-known variational problem related to the stability of Poiseuille flow in a flat cylinder. Amick left out the investigation of uniqueness and  existence of solutions with large flux. A rich and detailed analysis of the flow with large flux is due to Ladyzhenskaya and Solonnikov (\hspace{1sp}\cite{LS}) where existence of solutions to the steady Navier-Stokes system in nozzles were proved. However, the uniqueness and asymptotic far field behavior of the solutions obtained in \cite{LS} are not very clear.  There are lots of  further studies on the well-posedness for Leray problem and far field behavior for the associated solutions, one may refer to \cite{AmickF,MF, Rabier1, Rabier2, HW, LS, AP,Pileckas}, etc.
For more references on steady solutions of the Navier-Stokes equation in nozzles or other type of domains, please refer to the book by Galdi \cite{Galdi}. A significant open problem posed in \cite[p. 19]{Galdi} is the global well-posedness for Leray problem in a general nozzle when the flux $\Phi$ is large.

With the aid of the compactness of solutions obtained in \cite{LS} and the blowup method,
in order to get the global well-posedness for  Leray problem in a general nozzle tending to a pipe,  a key step is to  prove global uniqueness of Poiseuille flow in a pipe with the no slip boundary conditions. As a first step to study global uniqueness of Poiseuille flows, the local uniqueness was addressed in \cite{WX1,WX2}.  In fact,  the uniform structural stability of Poiseuille flows was established in \cite{WX1}, and it was even proved in \cite{WX2} that the solution is unique in a suitably large neighborhood of Hagen-Poiseuille flow when the flux is large. Furthermore, the solutions tend to the Hagen-Poiseuille flows exponentially fast as long as the external force tends to zero exponentially fast (\hspace{1sp}\cite{WX2}). For the Poiseuille flows in two dimensional infinitely long strip, the uniqueness of the solutions in the class of symmetric flows  was obtained in \cite{Rabier1}, while the uniqueness of the solutions in the class of general two dimensional flows was obtained only for the case with small flux (\hspace{1sp}\cite{Rabier2}).  The uniqueness and  uniform structural stability of two dimensional Poiseuille flows with any flux in a periodic strip was achieved in \cite{SWX} when the period is not very big.

On the other hand, the general boundary conditions for the system \eqref{NS} are the Navier boundary conditions
\be \label{NavierBC}
\Bu \cdot \Bn = 0, \ \ \ \ \ \ \ 2\Bn \cdot D(\Bu) \cdot \Bt +  \alpha \Bu \cdot \Bt = 0 \ \ \ \ \mbox{on}\ \partial \Omega,
\ee
where $\alpha \geq  0$ is called the slip coefficient which measures the tendency of a fluid to slip over the boundary and $D(\Bu) = \frac{\nabla \Bu + (\nabla \Bu)^t }{2}$ is the strain tensor. The boundary conditions \eqref{NavierBC} were first proposed by Navier (\hspace{1sp}\cite{Navier}) and were used  as an effective boundary condition for flows over rough boundaries via asymptotic and rigorous analysis in \cite{Achdou,JM}. Formally, as $\alpha \rightarrow \infty$, the Navier boundary conditions become the Dirichlet boundary conditions. In order to fix the solution of the problem \eqref{NS}-\eqref{NavierBC} in a nozzle, we require the following flux constraint
\be \label{flux}
\int_{\Sigma} \Bu \cdot \Bn \, dS = \Phi,
\ee
where $\Phi \in \mathbb{R}$ is called the flux of the flow. Without loss of generality, we always assume that $\Phi$ is nonnegative.

A typical straight cylinder in fluid mechanics and engineering is the circular pipe $\Omega = B_1(0) \times \mathbb{R}$. When $\BF=0$, the Navier-Stokes system \eqref{NS}-\eqref{flux} in $\Omega$ has an explicit solution $\Bu = (0, 0, u^z)$ with
\be \label{Poiseuille}
u^z = \bar{U}(r) = \frac{4+  2 \alpha}{4 +  \alpha } \left( 1 -  \frac{2\alpha}{4 +  2\alpha}r^2  \right) \frac{\Phi}{\pi}, \ \ \ \ \mbox{with}\ \ r= \sqrt{x^2 + y^2}.
\ee
Later on, we also call the solution $\bBU = \bar{U}(r) \Be_z $  Poiseuille flow.

It is also interesting to study the general Leray problem, i.e., to prove the existence of solutions in a general nozzle with Navier boundary conditions \eqref{NavierBC}, which converge to the Poiseuille flows at far fields where the nozzle tends to be straight.

However, the studies on steady Navier-Stokes system \eqref{NS} with Navier boundary conditions are not as many as that for the problem with no-slip boundary conditions. This is also the situation even for the problem in bounded domains.
The existence and regularity  of steady solutions of Navier-Stokes system \eqref{NS} supplemented with { homogeneous boundary conditions \eqref{NavierBC} in simply connected bounded domains} were obtained in \cite{AACG}.  { Moreover, uniform estimates with respect to $\alpha$ were also obtained in \cite{AACG}, however, the uniqueness is still unclear.} The  two-dimensional and three-dimensional axisymmetric solutions for  Navier-Stokes system with Navier boundary conditions in an infinitely long nozzle was investigated in \cite{Mucha1, Mucha2, Mucha3}  when the geometry of nozzles or the slip coefficient satisfies certain constraints. It was proved in \cite{SWX2} that there exists a solution for the Navier-Stokes system in a general two-dimensional nozzle. In order to solve the general Leray problem, one needs only to prove a Liouville type theorem for the general Poiseuille flows in a straight cylinder   with Navier boundary conditions. It was proved in \cite{WX-Navier} that under general Navier boundary conditions,  the Poiseuille flow is uniformly structural stable with respect to both the slip coefficient and the flux in a pipe periodic in the axial direction. This, in particular, implies the local uniqueness of Poiseuille flows in the pipe periodic in the axial direction. The aim of this paper is to study the uniqueness and uniform structural stability of Poiseuille flow in an infinitely long pipe.

In this paper, we study the problem \eqref{NS}-\eqref{NavierBC} in an infinitely long pipe, i.e., $\Omega= B_1(0) \times \mathbb{R}$. The aim is to prove the existence and uniqueness of solutions to \eqref{NS}-\eqref{flux} in a neighborhood of Poiseuille flow for every $\alpha$ and $\Phi$.  We  start with the wellposedness for the following linearized perturbation system
\be \label{linearizedNS}
\left\{
\ba & \bBU \cdot \nabla \Bv + \Bv \cdot \nabla \bBU - \Delta \Bv + \nabla P = \BF \ \ \ \ \mbox{in}\ \Omega, \\
& {\rm div}~\Bv = 0\ \ \ \ \mbox{in}\ \Omega,
\ea  \right.
\ee
supplemented
with the boundary conditions and the flux constraint,
\be \label{fluxBC}
\Bv \cdot \Bn =0,\ \ \ \ 2\Bn \cdot D(\Bv) \cdot \Bt + \alpha \Bv \cdot \Bt = 0\ \  \mbox{on} \ \partial \Omega, \ \ \ \ \ \ \int_{\Sigma} \Bv
\cdot \Bn \, dS = 0.
\ee

Our first main result is the following uniform estimates for the solutions of \eqref{linearizedNS}-\eqref{fluxBC}.
 \bt \label{thm1}
Assume that $\BF= \BF(r, z) \in L^2(\Omega)$  is axisymmetric. 

(a)\, If $F^\theta  = 0$, then the linear problem \eqref{linearizedNS}-\eqref{fluxBC} has a unique axisymmetric solution $\Bv \in H^2(\Omega)$, which satisfies $v^\theta = 0$ and
\be \label{estuniformlinear}
\|\Bv\|_{H^{1}(\Omega)} \leq C \|\BF\|_{L^2 (\Omega)},
\ee
\be \label{estimatelinear}
\|\Bv\|_{ H^2 (\Omega)} \leq C (1 + \Phi^{\frac14} ) \|\BF\|_{L^2 (\Omega)}.
\ee
Moreover, there exists a large constant $\Phi_0$ such that when $\Phi \geq \Phi_0$, it holds that
\be \label{estimatelinear-Philarge}
\|v^r \|_{L^2(\Omega)} \leq C \Phi^{-\frac45} \|\BF\|_{L^2(\Omega)}\ \ \ \ \text{and}\ \ \  \|\partial_z v^z \|_{L^2(\Omega)} \leq C \Phi^{-\frac37} \|\BF\|_{L^2(\Omega)},
\ee
where $C$ is a uniform constant independent of $\BF$, $\Phi$, and $\alpha$.

(b)\, If $F^\theta \neq 0$, $\alpha \geq \alpha_0 > 0$,  then the linear problem \eqref{linearizedNS}-\eqref{fluxBC} has a unique axisymmetric solution $\Bv \in H^2(\Omega)$  satisfying \eqref{estuniformlinear}-\eqref{estimatelinear}, and
\be \label{estimatelinear-swirl}
\|\Bv^\theta\|_{H^2(\Omega)} \leq C \left( 1+  \frac{1}{\alpha_0} \right) \|F^\theta \|_{L^2(\Omega)},  \ \ \ \ \|\partial_z v^\theta\|_{L^2(\Omega)} \leq C\left( 1+ \frac{1}{\alpha_0 } \right)^{\frac12} \Phi^{-\frac12} \|F^\theta \|_{L^2(\Omega)},
\ee
where $C $ is a uniform constant independent of $\BF$,  $\Phi$, and $\alpha$.
\et

\begin{remark}
It is noted that the estimates for $v^r$, $\partial_z v^z$, and $\partial_z v^\theta$ are even better when $\Phi$ is large. This is the key fact that helps to get the existence and uniquenss of solutions for nonlinear problem when  $\BF$ is the large in the case with large flux $\Phi$, see Part (b) of Theorem \ref{mainthm}.
\end{remark}

{
	\begin{remark}
	The condition $\alpha>0$ is needed because  the compatibility conditions are needed to guarantee the existence of solutions for the problem with $\alpha=0$. In fact, the associated homogeneous system has infinitely many nonzero solutions in the case $\alpha =0$. The condition for $\alpha>0$ is also needed to study steady Navier-Stokes system  with Navier boundary conditions in bounded domains (cf. \cite{AACG}).
\end{remark}
}
%\begin{remark}
%The key point of Theorem \ref{thm1} is that the constant $C$ in \eqref{estuniformlinear} does not depend on $\Phi$ or $\alpha$, which provides a uniform estimate for the solutions of the problem \eqref{linearizedNS}-\eqref{fluxBC} even when the coefficients of the system are very large.
%\end{remark}

With the aid of the uniform estimates for the linear system, we have the following results on the existence and local uniqueness of solutions to the problem \eqref{NS}-\eqref{flux}.
\bt \label{mainthm}
Assume that $\BF= \BF(r, z)\in L^2 (\Omega)$ is axisymmetric and $F^\theta = 0$.

(a)\, There exists a  constant $\veps_0$, independent of $\BF$, $\Phi$, and $\alpha$, such that if
\be \label{thmuniformnonlinear1}
\|\BF\|_{L^2  (\Omega)} \leq \veps_0,
\ee
then the steady Navier-Stokes system \eqref{NS} supplemented with the boundary conditions \eqref{NavierBC} and the flux constraint \eqref{flux} has a unique axisymmetric solution
$\Bu$ without swirl (i.e., $u^\theta = 0$), which satisfies  the estimates
\be \label{thmuniformnonlinear2}
\|\Bu - \bBU \|_{H^{1} (\Omega)} \leq C \|\BF\|_{L^2(\Omega)}
\ee
and
\be \label{thmuniformnonlinear3}
\|\Bu - \bBU \|_{H^2(\Omega)} \leq C (1 + \Phi^{\frac14}) \|\BF\|_{L^2(\Omega)}.
\ee

(b) There exists a constant $\Phi_1(>1)$ such that for every $\Phi \geq \Phi_1$, if
\be \nonumber
\|\BF\|_{L^2(\Omega)} \leq \Phi^{\frac{1}{40}},
\ee
the problem \eqref{NS}-\eqref{flux} has a unique axisymmetric solution satisfying $u^\theta = 0$ and the estimates \eqref{thmuniformnonlinear2}-\eqref{thmuniformnonlinear3},
\be \label{thmuniformnonlinear4}
\|\Bu^r \|_{H^{\frac54}(\Omega)} + \|\partial_z \Bu^z \|_{H^{\frac14} (\Omega) } \leq \Phi^{-\frac{1}{10} } \|\BF\|_{L^2(\Omega)},
\ \ \ \ \|\Bu^z \|_{H^{\frac54} (\Omega) } \leq 2C_1 \Phi^{\frac{1}{16}} \|\BF\|_{L^2(\Omega)}.
\ee
Here $C_1$ is a uniform constant independent of $\BF$,  $\Phi$, and $\alpha$.
\et

Furthermore, if $F^\theta \neq 0$, we have the  the following results.

\bt \label{mainthm2}
Assume that $\BF = \BF(r, z) \in L^2(\Omega)$ is axisymmetric and $\alpha \geq \alpha_0 > 0$ where $\alpha_0$ is any fixed positive constant.

(a)\ There exists a constant $\epsilon_0$, independent of $\BF$, $\Phi$, and $\alpha$, such that if
\be \nonumber
\|\BF\|_{L^2(\Omega)} \leq \epsilon_0,
\ee
the problem \eqref{NS}-\eqref{flux} has a unique axisymmetric solution $\Bu$ satisfying
\eqref{thmuniformnonlinear2}-\eqref{thmuniformnonlinear3}.

(b)\ There exists a  constant $\Phi_2(>1)$ depending only on $\alpha_0$,  such that for every $\Phi \geq \Phi_2$, if
\be \nonumber
\|\BF\|_{L^2(\Omega)} \leq \Phi^{\frac{1}{40}},
\ee
the problem \eqref{NS}-\eqref{flux} has a unique axisymmetric solution satisfying the estimates
\eqref{thmuniformnonlinear2}-\eqref{thmuniformnonlinear4}, and
\be  \label{thmuniformnonlinear5}
\|\Bu^\theta\|_{H^{\frac54}(\Omega)} \leq 2C_2\left( 1 + \frac{1}{\alpha_0} \right) \|\BF\|_{L^2(\Omega)} , \ \ \ \ \ \|\partial_z  \Bu^\theta \|_{H^{\frac14}(\Omega)}
\leq \Phi^{-\frac14}  \|\BF\|_{L^2(\Omega)},
\ee
where $C_2 $ is a  uniform constant independent of $\BF$, $\Phi$, and  $\alpha_0$.
\et

\begin{remark}
Following almost the same proof as in \cite{WX2}, one can prove the exponential convergence of the solution to the Poiseuille flow as $|z|\to \infty$, when the external force decays exponentially at far fields.
\end{remark}

%\begin{remark}
%	The vanishing viscocity of unsteady solutions of Navier-Stokes system with Navier boundary conditions was investigated in \cite{WWX}.  The uniform estimate obtained in Theorems \ref{mainthm} and \ref{mainthm2} should be also helpful for the analysis on the boundary layers for the flows as the slip coefficient changes.
%\end{remark}

{ There are many studies on unsteady flows with Navier boundary conditions in recent years. Here we just mention a few of them which relate to the problem studied in this paper. The stability of zero solution in a strip with Navier boundary conditions  and enhanced dissipation of Poiseuille flows in a two dimensional strip under total slip conditions are investigated in \cite{Ding1} and \cite{Ding2}, respectively. The stability of Couette flows of compressible Navier-Stokes system with Navier boundary conditions was studied in \cite{Li}. Furthermore, the vanishing viscosity limit of unsteady flows with different slip coefficients was analyzed in \cite{WWX}.}

{The essential ideas in this paper are similar to that in \cite{WX-Navier}, but the details are much more involved. Here we give some of these key ideas. First, the stream function for the axisymmetric flows satisfies a fourth order equation, which can be transformed to a fourth order ODE with a frequency parameter after taking Fourier transform in the axial direction. With the aid of energy estimate, one can get the estimate for the stream function when the flux is small or large but with low or high frequency. With the aid of the construction of boundary layers, the case with large flux, intermediate frequency, and small (large) slip coefficient can also be handled where uniform $H^2$ estimate can be obtained. For the problem with large flux, intermediate frequency, and intermediate slip coefficient, we  get weaker estimate, i.e., uniform $H^1$ estimate. However, these estimates are enough to study nonlinear problem together with some better estimate for the quantities like $v^r$, $\partial_z v^z$, etc.}

The organization of the rest of the  paper is as follows. In Section \ref{Linear},  the stream function formulation for both nonlinear problem  \eqref{NS}-\eqref{flux} and  linearized problem \eqref{linearizedNS}-\eqref{fluxBC} in the axisymmetric case is introduced. Some uniform a priori estimates for the stream function of the linearized problem are given in Section \ref{sec-res}.  The uniform a priori estimate of the swirl velocity is established in Section \ref{sec-swirl}. With the aid of the analysis on the associated linearized problem and iteration method, the uniform nonlinear structural stability of Poiseuille flows in axisymmetric case is proved in Section \ref{secnonlinear}. The appendix collects  some important lemmas which are used here and there in the paper.

%%%%%%%%%%%%%%%%%%%%%%%%%%%%%%%%%%%Section2%%%%%%%%%%%%%%%%%%%%%%%%%%%%%%%%%%%%%%%%%%%%%%%%%%%%%%%%%%%%%%%%%%%%%%%%%%%%%%%%%%%

\section{Stream function formulation }\label{Linear}
This section devotes to the stream formulation  of the nonlinear and   linear systems of Navier-Stokes equations.  An important observation is that the equation for swirl velocity decouples from the equations for  radial and axial velocity. After introducing the stream function,  the equations for axial and radial velocity can be reduced into a single fourth order equation.

\subsection{Stream function formulation for nonlinear perturbation system}
Suppose that $\Bu$ is an axisymmetric solution of \eqref{NS}-\eqref{flux}, $\bBU= \bar{U} \Be_z$ is the Poiseuille flow with $\bar{U}$ defined in \eqref{Poiseuille}. Let
\be \nonumber
\Bv= \Bu - \bBU = v^r (r, z) \Be_r + v^\theta (r, z) \Be_\theta + v^z(r, z)\Be_z.
\ee

In terms of cylindrical coordinates, the system  for perturbation  $\Bv$ can be written as the follows,
\be \label{nonlinearperturb}
\left\{ \ba
& \bar{U}(r) \frac{\partial v^r}{\partial z} + \frac{\partial P}{\partial r} - \left[ \frac{1}{r} \frac{\partial }{\partial r} \left( r \frac{\partial v^r}{\partial r} \right) + \frac{\partial^2 v^r}{\partial z^2} - \frac{v^r}{r^2}  \right] = F^r - (v^r \partial_r v^r + v^z \partial_z v^r ) + \frac{(v^\theta)^2 }{r}\ \ \mbox{in}\ D, \\
& \bar{U}(r) \frac{\partial v^z }{\partial z } +v^r \frac{\partial \bar{U}}{\partial r} + \frac{\partial P }{\partial z } - \left[ \frac{1}{r} \frac{\partial }{\partial r} \left( r \frac{\partial v^z}{\partial r} \right) + \frac{\partial^2 v^z }{\partial z^2}  \right] = F^z - (v^r \partial_r v^z + v^z \partial_z v^z)\ \ \ \mbox{in} \ D,\\
& \bar{U}(r) \frac{\partial v^\theta}{\partial z} - \left[ \frac1r \frac{\partial}{\partial r} \left( r \frac{\partial v^\theta}{\partial r } \right) + \frac{\partial^2 v^\theta}{\partial z^2} - \frac{v^\theta}{r^2} \right] = F^\theta - (v^r \partial_r v^\theta + v^z \partial_z v^\theta) - \frac{v^r v^\theta}{r}\ \ \ \ \mbox{in}\ D,   \\
& \partial_r v^r + \partial_z v^z + \frac{v^r}{r} = 0\ \ \ \mbox{in} \ D.
\ea  \right.
\ee
Here $F^r$, $F^z$, and $F^\theta$ are the radial, axial, and azimuthal component of $\BF$, respectively, and $D=\{(r, z): r\in (0, 1), z\in \mathbb{R}\}$.
The  boundary conditions and the flux constraint  can be written as
\be \label{BC-1}
v^r(1, z)  = 0,\quad  \partial_z v^r -\partial_r v^z = \alpha v^z,\quad  \int_0^1 r v^z(r, z)\, dr = 0,
\ee
and
\begin{equation}\label{BC-swirl}
\frac{ \partial v^\theta}{\partial r} (1, z) = (1 - \alpha) v^\theta(1, z).
\end{equation}

It follows from  the fourth equation in \eqref{nonlinearperturb} that there exists a stream function $\psi(r, z)$ satisfying
\be \label{2-0-4}
v^r =  \partial_z \psi \ \ \text{and} \ \ v^z = - \frac{\partial_r (r \psi) }{r}.
\ee
Then the azimuthal vorticities of $\Bv$ and $\BF$ are written as
\be \nonumber
\omega^\theta= \partial_z v^r - \partial_r v^z= \frac{\partial }{\partial r}  \left(  \frac1r \frac{\partial }{\partial r} (r \psi) \right) + \partial_z^2 \psi \ \ \ \ \mbox{and}\ \ \ \ f= \partial_z F^r - \partial_r F^z,
\ee
respectively.
Taking the first two equations in \eqref{nonlinearperturb} yields that
\be \label{2-0-2}
\bar U(r)  \partial_z \omega^\theta - \left(\partial_r^2 + \partial_z^2 + \frac{1}{r}  \partial_r \right)\omega^\theta
+ \frac{\omega^\theta}{r^2}  = \partial_z F^r - \partial_r F^z - \partial_r (v^r \omega^\theta) - \partial_z (v^z \omega^\theta) +
\partial_z \left[ \frac{(v^\theta)^2}{r} \right].
\ee
Therefore, the stream function $\psi$ satisfies the following fourth order equation,
\be \label{2-0-4-1}
\bar U(r)  \partial_z( \mathcal{L} + \partial_z^2) \psi -
( \mathcal{L} + \partial_z^2)^2 \psi =\partial_z F^r - \partial_r F^z - \partial_r (v^r \omega^\theta) - \partial_z (v^z \omega^\theta) +
\partial_z \left[ \frac{(v^\theta)^2}{r} \right],
\ee
where
\be \nonumber
\mathcal{L} = \frac{\partial}{\partial r} \left(  \frac1r \frac{\partial}{\partial r}(r \cdot)      \right) = \frac{\partial^2}{\partial r^2} + \frac1r \frac{\partial}{\partial r} - \frac{1}{r^2}.
\ee
 Next, we derive the boundary conditions for $\psi$. As discussed in \cite{Liu-Wang}, in order to get classical solutions, some compatibility conditions at the axis should be imposed. Assume that the velocity $\Bv$ and the vorticity $\Bo$ are continuous, then $v^r(0, z)$ and $\omega^\theta(0, z)$ should vanish. This implies that
\be \nonumber
\partial_z \psi(0, z) = (\mathcal{L} + \partial_z^2)\psi (0, z) = 0.
\ee
Hence, without loss of generality, one can  assume that $\psi(0, z) = 0$, and  the following compatibility condition holds at the axis,
\be \label{2-0-4-2}
\psi(0, z) = \mathcal{L} \psi(0, z) = 0.
\ee
On the other hand, it follows from \eqref{BC-1} that
\be \nonumber
\int_0^1 \partial_r (r \psi ) (r, z)\, dr =-  \int_0^1 r v^z \, dr = 0.
\ee
This, together with \eqref{2-0-4-2}, gives
\be \label{2-0-4-3}
\psi(1, z) = \lim_{r\rightarrow 0+ } ( r \psi)  (r, z)  = 0.
\ee
Moreover, according to the Navier boundary condition \eqref{fluxBC} for $\Bv$, one has
\be \nonumber
\omega^\theta =   \alpha v^z \quad \text{at}\,\, r=1.
\ee
This implies that
\be \nonumber
( \mathcal{L} + \partial_z^2 ) \psi = -\alpha \frac{\partial_r (r \psi )}{r } \ \ \ \ \mbox{at}\ \ r=1 .
\ee
Note that $\psi(1, z) = 0$, hence
\be \label{2-0-4-5}
\mathcal{L} \psi(1, z) = -\alpha \frac{\partial }{\partial r} \psi (1, z) .
\ee

\subsection{Stream function formulation for linear perturbation problem}
In terms of cylindrical coordinates, the linearized perturbation equations \eqref{linearizedNS} for the axisymmetric solutions become
\be \label{2-0-1-1}
\left\{
\ba
& \bar U(r)  \frac{\partial v^r}{\partial z} + \frac{\partial P}{\partial r} -\left[ \frac{1}{r} \frac{\partial}{\partial r}\left(
r \frac{\partial v^r}{\partial r} \right) + \frac{\partial^2 v^r}{\partial z^2} - \frac{v^r}{r^2} \right] = F^r  \ \ \ \mbox{in}\ D ,\\
& v^r \frac{\partial \bar U }{\partial r} + \bar U(r)  \frac{\partial v^z}{\partial z} + \frac{\partial P}{\partial z}
- \left[ \frac{1}{r} \frac{\partial }{\partial r} \left( r \frac{\partial v^z}{\partial r}\right) + \frac{\partial^2 v^z}{\partial z^2} \right] = F^z \ \ \mbox{in}\ D ,                     \\
& \partial_r v^r + \partial_z v^z + \frac{v^r}{r} =0\ \ \ \mbox{in}\ D
\ea \right.
\ee
and
\be \label{vswirl}
\bar U(r)  \partial_z  v^\theta - \left[ \frac{1}{r} \frac{\partial }{\partial r} \left( r \frac{\partial v^\theta}{\partial r}\right) + \frac{\partial^2 v^\theta}{\partial z^2} - \frac{v^\theta}{r^2} \right] =  F^\theta \ \ \ \mbox{in}\ \ D.
\ee

Similarly, one can introduce the stream function $\psi$ for the solution $(v^r, v^z)$ of \eqref{2-0-1-1}. Then $\psi$ satisfies the following fourth order equation,
\be \label{stream-formulation}
\bar{U}(r) \partial_z (\mathcal{L} + \partial_z^2) \psi - (\mathcal{L} + \partial_z^2)^2 \psi = \partial_z F^r - \partial_r F^z = f.
\ee
Furthermore, the boundary conditions for $\psi$ are of the form
\be \label{stream-BC}
\mathcal{L} \psi(0, z) = \psi(0, z) = \psi(1, z) =0, \ \ \ \ \mathcal{L}\psi(1, z) = - \alpha \frac{\partial \psi}{\partial r }(1, z) .
\ee

At last, let us introduce some notations. For a given function $g(r, z)$, define its Fourier transform with respect to $z$ variable by
\be  \nonumber
\hat{g}(r, \xi) = \int_{\mathbb{R}} g(r, z) e^{-i \xi z} dz.
\ee
Let $\Re g$ and $\Im g$ denote the real and imaginary part of a function or a number $g$, respectively.

\section{Uniform estimate independent for stream function }\label{sec-res}
The main goal of this section is  to prove part (a) of Theorem \ref{thm1}, i.e., the uniform estimates for solutions to the linear problem \eqref{stream-formulation}-\eqref{stream-BC} and the corresponding velocities. The main idea is in the same spirit as that in \cite{WX-Navier}, where we deal with the solutions periodic in the axial direction.

Taking the Fourier transform with respect to $z$ for the equation \eqref{2-0-4-1}, for each fixed $\xi$, $\hat{\psi}$ satisfies
\be \label{2-0-8}
i \xi \bar{U}(r) ( \mathcal{L} -\xi^2) \hat{\psi} - ( \mathcal{L} -\xi^2)^2 \hat{\psi} = \hat{f}.
\ee
The boundary conditions \eqref{2-0-4-2}-\eqref{2-0-4-5} can be written as
\be \label{FBC}
 \hat{\psi}(0) = \hat{\psi} (1) = \mathcal{L} \hat{\psi} (0 ) = 0\ \ \ \mbox{and}\ \ \mathcal{L}\hat{\psi}(1) =-  \alpha \hat{\psi}^{\prime}(1).
\ee
{In what follows, first we give the uniform estimate for $\psi$ with respect to $\alpha$, when $\Phi$ is not large. Then we deal with the case with large flux in terms of three subcases. More precisely,  choosing some small constant $\epsilon_1 \in (0, 1)$, the proof is divided into three subcases: (1)\ $ |\xi| \leq \frac{1}{\epsilon_1 \Phi}$ \ (2)\ $|\xi | \geq \epsilon_1 \sqrt{\Phi} $ \ (3) \ $\frac{1}{\epsilon_1 \Phi} <  |\xi| < \epsilon_1 \sqrt{\Phi} $. In the first two cases, the equation  \eqref{2-0-8} is essentially a biharmonic equation so that the standard energy estimates are enough. While the elaborate boundary layer analysis is used to deal with the last case when slip coefficients are big or small comparing with $(\Phi|\xi|)^{1/3}$. In these two cases, the analysis is quite similar to the problem with $\xi$ as integers, i.e., the solutions are periodic in the axial direction. However, when the slip coefficients are comparable with $(\Phi|\xi|)^{1/3}$, then the uniform estimate for $\|(v^r, v^z)\|_{H^1}$and some other better estimate such as $\|v^r\|_{L^2}$ are obtained via delicate analysis on the problem.
}

\subsection{Uniform A priori estimate for stream function when $\Phi$ is not large}
In this subsection,  some basic estimates for the solutions of the problem \eqref{2-0-8}--\eqref{FBC} are established. These estimates are uniform with respect to $\Phi$ and $\alpha$, when $\Phi$ is not large.
\begin{pro}\label{smallflux}
Let $\hat{\psi}$ be a smooth solution to \eqref{2-0-8}-\eqref{FBC}. It holds that
\be \la{stream-smallflux}
\ba
  & \int_0^1\left( |\mathcal{L} \hat{\psi}|^2 r \, + \xi^2  \left| \frac{d}{dr}(r \hat{\psi} )  \right|^2 \frac1r \, + \xi^4  |\hat{\psi}|^2 r \right)\, dr + \alpha \left|  \frac{d}{dr} (r \hat{\psi}) (1)  \right|^2 \\
\leq  & C (1+ \Phi^2) \int_0^1  (|\widehat{F^r} |^2 + |\widehat{F^z}|^2 ) r \, dr .
\ea
\ee
\end{pro}

\begin{proof}
Multiplying \eqref{2-0-8} by $r \overline{\hat{\psi}}$ and integrating the resulting equation over $[0, 1]$ yield
\be \label{3-1}
\int_0^1 \left[ i \xi \bar{U} (r) (\mathcal{L} - \xi^2) \hat{\psi} - (\mathcal{L} - \xi^2)^2 \hat{\psi} \right] \overline{\hat{\psi}} r \, dr
= \int_0^1 \hat{f} \overline{\hat{\psi} } r\, dr .
\ee
For the first term on the left hand of \eqref{3-1}, it follows from integration by parts and the homogeneous boundary conditions for $\hat{\psi}$ that
\be \label{3-2}
 \ba
& \int_0^1 i \xi  \bar U(r)  ( \mathcal{L}  - \xi^2) \hat{\psi}    \overline{\hat{\psi} }  r \, dr \\
= \,\,& i \xi \int_0^1 \bar{U}(r)  \frac{d}{dr} \left(  \frac1r \frac{d}{dr} (r \hat{\psi} )     \right)  r \overline{ \hat{\psi} } \, dr
-  i \xi^3 \int_0^1 \bar U(r) | \hat{\psi} |^2  r \, dr \\
= \,\,&   i \xi \frac{4\Phi}{\pi} \frac{\alpha}{4+ \alpha }  \int_0^1 \frac{d}{dr} (r \hat{\psi} )  r \overline{ \hat{\psi} }\, dr
- i \xi \int_0^1 \frac{\bar{U}(r)}{r} \left| \frac{d}{dr} (r \hat{\psi} )   \right|^2 \, dr - i \xi^3 \int_0^1 \bar{U}(r) |\hat{\psi} |^2  r \, dr.
\ea \ee
 While for the second term on the left hand  of \eqref{3-1}, one has
\be \la{3-3} \ba
& \int_0^1 ( \mathcal{L}  - \xi^2)^2 \hat{\psi}  \overline{\hat{\psi} }  r \, dr \\
 =\,\, & \int_0^1 \frac{d}{dr} \left(  \frac1r \frac{d}{dr} ( r \mathcal{L} \hat{\psi} )      \right)  \overline{\hat{\psi} }  r \, dr
- 2 \xi^2 \int_0^1 \frac{d}{dr} \left( \frac1r \frac{d}{dr} (r \hat{\psi}  )    \right)  \overline{ \hat{\psi} }  r \, dr
+ \xi^4 \int_0^1 |\hat{\psi}|^2  r\, dr \\
 =\,\,& \int_0^1 | \mathcal{L} \hat{\psi}  |^2  r \, dr - \mathcal{L} \hat{\psi} (1) \frac{d}{dr} ( r \overline{\hat{\psi} } )(1)  + 2 \xi^2 \int_0^1 \left|  \frac{d}{dr} (r \hat{\psi} )  \right|^2  \frac1r \, dr
+ \xi^4 \int_0^1 | \hat{\psi} |^2  r \, dr\\
= \,\, &  \int_0^1 | \mathcal{L} \hat{\psi} |^2  r \, dr + \alpha \left| \frac{d}{dr}(r \hat{\psi} )(1) \right|^2   + 2 \xi^2 \int_0^1 \left|  \frac{d}{dr} (r \hat{\psi} )  \right|^2  \frac1r \, dr
+ \xi^4 \int_0^1 |\hat{\psi} |^2  r \, dr.
\ea \ee

It follows from  \eqref{3-1}-\eqref{3-3} that
\be \la{3-5} \ba
& \int_0^1 | \mathcal{L} \hat{\psi} |^2  r \, dr
+ 2 \xi^2 \int_0^1 \left|  \frac{d}{dr} (r \hat{ \psi}  )   \right|^2  \frac1r \, dr
+ \xi^4 \int_0^1 |\hat{\psi}  |^2  r \, dr +  \alpha \left| \frac{d}{dr}(r \hat{\psi}  )(1) \right|^2  \\
 =\,\, &- \Re \int_0^1 \hat{f} \overline{ \hat{\psi} }  r \, dr -  \frac{4 \Phi}{\pi} \frac{\alpha}{4 + \alpha } \xi  \Im \int_0^1  \left[ \frac{d}{dr} (r \hat{\psi} )  r \overline{\hat{\psi} }      \right] \, dr
\ea \ee
and
\be \la{3-6} \ba
&  \xi  \int_0^1 \frac{\bar{U}(r) }{ r } \left| \frac{d}{dr} ( r \hat{\psi} )  \right|^2 \, dr
 + \xi^3 \int_0^1  \bar{U}(r) | \hat{\psi} |^2  r\, dr
 -  \frac{4 \Phi}{\pi} \frac{\alpha}{ 4+ \alpha } \xi \Re \int_0^1   \left[ \frac{d}{dr} ( r \hat{\psi}  )  r \overline{ \hat{\psi}  } \right] \, dr\\
 & = - \Im \int_0^1 \hat{f}  \overline{ \hat{\psi} }  r \, dr.
\ea \ee
The homogeneous boundary conditions for $\hat{\psi} $ imply
\be \nonumber
\Re \int_0^1 \frac{d}{dr} (r \hat{\psi}  ) r \overline{ \hat{\psi} } \, dr  = 0.
\ee
Hence the expression  \eqref{3-6} can be rewritten as
\be \la{3-7}
 \xi \int_0^1 \frac{\bar{U}(r) }{ r } \left| \frac{d}{dr} ( r \hat{\psi} )   \right|^2 \, dr
+ \xi^3 \int_0^1  \bar{U}(r) | \hat{\psi}  |^2  r\, dr
 = - \Im \int_0^1 \hat{f } \overline{\hat{\psi}  }  r \, dr .
\ee

Since $\alpha \geq  0$, for every $0 \leq r < 1$,
\be \label{3-7-1}
\frac{\frac{2\Phi}{\pi} (1 - r^2) }{\bar{U} (r) }
= \frac{2(1-r^2)}{\frac{4 + 2\alpha}{4 + \alpha} \left( 1 - \frac{2\alpha}{4 + 2\alpha } r^2 \right)   }
\leq \frac{2(1-r^2)}{\frac{4 + 2\alpha}{4 + \alpha} \left( 1 - r^2 \right)   } \leq 2.
\ee
It follows from \eqref{3-7}-\eqref{3-7-1} and Lemma \ref{lemmaHLP}  that
\be \label{3-8} \ba
\Phi |\xi| \int_0^1 |\hat{\psi}|^2 r \, dr
& \leq C \Phi |\xi| \int_0^1 \frac{1 - r^2 }{r} \left| \frac{d}{dr}(r \hat{\psi} ) \right|^2 \, dr  \\
& \leq C \left( \int_0^1 | \widehat{F^r} | | \xi \hat{\psi}| r \, dr + \int_0^1 |\widehat{F^z} | \left| \frac{d}{dr} ( r \hat{\psi} ) \right| \, dr \right).
\ea
\ee
Similarly, the estimate \eqref{3-5} gives
\be \la{3-9} \ba
&  \int_0^1 |\mathcal{L} \hat{\psi}|^2 r \, dr + 2\xi^2 \int_0^1 \left| \frac{d}{dr}( r \hat{\psi})    \right|^2 \frac1r \, dr + \xi^4 \int_0^1 |\hat{\psi} |^2 r \, dr +  \alpha \left| \frac{d }{dr} (r \hat{\psi} ) (1) \right|^2  \\
\leq & \int_0^1 | \widehat{F^r} | | \xi \hat{\psi}| r \, dr + \int_0^1 |\widehat{F^z} | \left| \frac{d}{dr} ( r \hat{\psi} ) \right| \, dr +
\frac{4\Phi}{\pi} |\xi| \int_0^1 \left| \frac{d}{dr} (r \hat{\psi})  \right| |r \hat{\psi}| \, dr .
\ea
\ee
According to \eqref{3-8}, one has
\be \la{3-11}
\ba
 & \frac{4\Phi}{\pi} |\xi | \int_0^1 \left| \frac{d}{dr} (r \hat{\psi})  \right| |r \hat{\psi}| \, dr \\
\leq & C \Phi^{\frac12} \left[ |\xi|  \int_0^1 \left| \frac{d}{dr}(r \hat{\psi} )  \right|^2 \frac1r \, dr   \right]^{\frac12} \left( \Phi |\xi| \int_0^1 |\hat{\psi}|^2 r \, dr  \right)^{\frac12} \\
\leq & \frac14 ( 1  + \xi^2) \int_0^1 \left| \frac{d}{dr}(r \hat{\psi} )  \right|^2 \frac1r \, dr + C \Phi \left( \int_0^1 | \widehat{F^r} | | \xi \hat{\psi}| r \, dr + \int_0^1 |\widehat{F^z} | \left| \frac{d}{dr} ( r \hat{\psi} ) \right| \, dr \right).
\ea
\ee
Substituting \eqref{3-11} into \eqref{3-9} and applying Lemma \ref{lemma1} and Cauchy-Schwarz inequality give
\be \la{3-12} \ba
  & \int_0^1\left( |\mathcal{L} \hat{\psi}|^2 r \, + \xi^2  \left| \frac{d}{dr}(r \hat{\psi} )  \right|^2 \frac1r \, + \xi^4  |\hat{\psi}|^2 r \right)\, dr + \alpha \left|  \frac{d}{dr} (r \hat{\psi}) (1)  \right|^2 \\
\leq  & C (1+ \Phi^2) \int_0^1  (|\widehat{F^r} |^2 + |\widehat{F^z}|^2 ) r \, dr,
\ea
\ee
which is exactly \eqref{stream-smallflux}. Hence the proof of Proposition \ref{smallflux} is completed.
\end{proof}

With the aid of this a priori estimate, for each fixed $\xi$, the existence of strong solutions to the problem  \eqref{2-0-8}--\eqref{FBC} has been established in \cite{WX-Navier}.

\begin{pro}
The corresponding velocity field $\Bv^*= v^r\Be_r + v^z \Be_z$ satisfies
\be
\|\Bv^*\|_{H^2(\Omega)} \leq C (1 + \Phi^2) \|\BF^*\|_{L^2(\Omega)}.
\ee
\end{pro}
\begin{proof}
Integrating \eqref{3-12} with respect to $\xi$ yields
\be \la{3-13} \ba
& \int_{-\infty}^{+\infty} \int_0^1 \left\{\left( |\mathcal{L} \hat{\psi}|^2 + \xi^4 |\hat{\psi}|^2 \right)r  \,
+  \xi^2 \left|  \frac{\partial }{\partial r}( r \hat{\psi})   \right|^2 \frac{1}{r} \right\}\, dr d\xi +  \alpha \int_{-\infty}^{+ \infty} \left|  \frac{\partial}{\partial r} (r \hat{\psi}) (1, \xi )  \right|^2 \, d\xi
 \\
\leq &  C(1 + \Phi^2) \| \BF^* \|_{L^2(\Omega)}^2 .
\ea
\ee
Thus one has
\be \la{3-15} \ba
\|\Bv^* \|_{H^1 (\Omega)}^2  & = \|\nabla \Bv^* \|_{L^2(\Omega)}^2 + \|\Bv^*\|_{L^2(\Omega)}^2\\
&=  \int_{\Omega} {\rm curl}~(\omega^\theta \Be_\theta) \cdot \Bv_n^* \, dV + \int_{\partial \Omega} \frac{\partial v^z }{\partial r} \cdot v^z \, dS
+ \|\Bv^*\|_{L^2(\Omega)}^2 \\
& = \int_{\Omega} |\omega^\theta|^2 \, dV + \alpha \int_{\partial \Omega} |v^z|^2 \, dS - \alpha \int_{\partial \Omega} |v^z|^2 \, dS + \|\Bv^*\|_{L^2(\Omega)}^2 \\
& =   \|\Bo^\theta \|_{L^2(\Omega)}^2 +    \|\Bv^* \|_{L^2(\Omega)}^2   \\
 & \leq  \int_{-\infty}^{+ \infty} \int_0^1 | (\mathcal{L} - \xi^2) \hat{\psi}|^2 r \, dr d\xi +  \int_{-\infty}^{+\infty} \int_0^1 \left| \frac{\partial (r \hat{\psi}) }{\partial r}  \right|^2 \frac1r + \xi^2 |\hat{\psi}|^2 r \, dr d\xi  \\
& \leq C  (1 + \Phi^2) \|\BF^* \|_{L^2(\Omega)}^2 .
\ea \ee

Multiplying \eqref{3-9} and  \eqref{3-11} by $\xi^2$, respectively, yields
\be \label{3-21} \ba
&  \xi^2 \int_0^1 |\mathcal{L} \hat{\psi}|^2 r \, dr + 2\xi^4 \int_0^1 \left| \frac{d}{dr}( r \hat{\psi})    \right|^2 \frac1r \, dr + \xi^6 \int_0^1 |\hat{\psi} |^2 r \, dr + \alpha  \xi^2 \left| \frac{d }{dr} (r \hat{\psi} ) (1) \right|^2  \\
\leq & \int_0^1 | \widehat{F^r} | | \xi^3 \hat{\psi}| r \, dr + \int_0^1 |\widehat{F^z} | \left| \xi^2 \frac{d}{dr} ( r \hat{\psi} ) \right| \, dr +
\frac{4\Phi}{\pi} |\xi|^3 \int_0^1 \left| \frac{d}{dr} (r \hat{\psi})  \right| |r \hat{\psi}| \, dr
\ea
\ee
and
\be \label{3-22} \ba
& \frac{4\Phi}{\pi} |\xi|^3 \int_0^1 \left| \frac{d}{dr} (r \hat{\psi} ) \right| |r \hat{\psi}| \, dr \\
\leq & \frac14 (\xi^2 + \xi^4) \int_0^1 \left| \frac{d}{dr} (r \hat{\psi} )  \right|^2 \frac1r \, dr + C \Phi \left( \int_0^1 | \widehat{F^r} | | \xi^3 \hat{\psi}| r \, dr + \int_0^1 |\widehat{F^z} | \left| \xi^2  \frac{d}{dr} ( r \hat{\psi} ) \right| \, dr \right).
\ea
\ee
Taking \eqref{3-22} into \eqref{3-21} and using Cauchy-Schwarz inequality give
\be \label{3-23}
\ba
  & \xi^2 \int_0^1\left( |\mathcal{L} \hat{\psi}|^2 r \, + \xi^4  \left| \frac{d}{dr}(r \hat{\psi} )  \right|^2 \frac1r \, + \xi^6  |\hat{\psi}|^2 r \right)\, dr + \alpha \xi^2 \left|  \frac{d}{dr} (r \hat{\psi}) (1)  \right|^2 \\
\leq  & C (1+ \Phi^2) \int_0^1  (|\widehat{F^r} |^2 + |\widehat{F^z}|^2 ) r \, dr ,
\ea
\ee
Similar to the proof for the estimate \eqref{3-15}, one has
\be \label{3-25}
\| \partial_z \Bv^*\|_{H^1(\Omega)} \leq C (1 + \Phi) \|\BF^*\|_{L^2(\Omega)}.
\ee

Note that  $\Bv^*$ satisfies the equation
\be \label{3-26} \left\{  \ba
& \bar{U} \partial_z  \Bv^* + v^r  \partial_r \bBU -\Delta \Bv^*  + \nabla P = \BF^* \ \ \ \mbox{in}\ \ \Omega, \\
& {\rm div}~\Bv^* = 0\ \ \ \ \mbox{in} \ \ \Omega.
\ea  \right.
\ee
According to the regularity theory for Stokes equations (\hspace{1sp}\cite[Lemma VI.1.2]{Galdi}) and the trace theorem for axisymmetric functions, one has
\be \label{3-27} \ba
\|\Bv^*\|_{H^2 (\Omega)} & \leq C \left(\|\BF^*\|_{L^2(\Omega)} +  \Phi \|\partial_z \Bv^*\|_{L^2(\Omega)} +  \Phi \|v^r\|_{L^2(\Omega)} +  \|\Bv^*\|_{H^1(\Omega)} + \|\Bv^* \|_{H^{\frac32}(\partial \Omega) } \right) \\
& \leq C (1 + \Phi^2 ) \|\BF^*\|_{L^2(\Omega)} + C \|\partial_z \Bv^*\|_{H^1(\Omega)} \\
& \leq C (1 + \Phi^2 ) \|\BF^*\|_{L^2(\Omega)} .
\ea
\ee
This finishes the proof of the proposition.
\end{proof}

%%%%%%%%%%%%%%%%%%%%%%%%%%%%%%%%%%%Stream function%%%%%%%%%%%%%%%%%%%%%%%%%%%%%%%%%%%%%

\subsection{Uniform estimate for the case with large flux and low frequency}
In this subsection,  the uniform estimate for the solutions of \eqref{2-0-8}--\eqref{FBC} with respect to $\Phi$ and $\alpha$ is obtained  when the flux is large and the frequency is low. We assume that $\Phi >1$ in this subsection.
\begin{pro}\label{Bpropcase1}
Assume that $|\xi| \leq \frac{1}{\epsilon_1 \Phi}\leq 1 $. Let $\hat{\psi}(r, \xi)$ be a smooth solution of the problem
\eqref{2-0-8}--\eqref{FBC}, then one has
\be \label{3-31}
\int_0^1 |\mathcal{L} \hat{\psi} |^2 r \, dr + 2 \xi^2 \int_0^1 \left| \frac{d}{dr}(r \hat{\psi})   \right|^2 \frac{1}{r}\, dr
+ \xi^4 \int_0^1 |\hat{\psi}|^2 r\, dr \leq C (\epsilon_1) \int_0^1 |\what{\BF^*}|^2 r \, dr.
\ee
\end{pro}

\begin{proof}
It follows from \eqref{3-5}, \eqref{3-7}, and  Lemma \ref{lemmaHLP} that one has
\be \label{3-33} \ba
& \Phi |\xi| \int_0^1 |\hat{\psi}|^2 r\, dr + \Phi |\xi| \int_0^1 \frac{1-r^2}{r} \left|  \frac{d}{dr} ( r \hat{\psi}) \right|^2 \frac1r \, dr + \Phi |\xi|^3 \int_0^1 (1 - r^2) |\hat{\psi}|^2 r \, dr \\
\leq & C \left( \int_0^1 |\what{\BF^*}|^2 r \, dr \right)^{\frac12} \left( \int_0^1 \xi^2 |\hat{\psi}|^2 r \, dr + \int_0^1 \left| \frac{d}{dr}(r \hat{\psi} ) \right|^2 \frac1r \, dr   \right)^{\frac12}
 \ea \ee
and
\be \label{3-34} \ba
& \int_0^1 |\mathcal{L} \hat{\psi}|^2 r \, dr + 2\xi^2 \int_0^1 \left| \frac{d}{dr}( r \hat{\psi} )  \right|^2 \frac1r \, dr + \xi^4 \int_0^1 |\hat{\psi}|^2 r \, dr  + \alpha \left| \frac{d}{dr} ( r\hat{\psi} ) (1)  \right|^2  \\
\leq &  C \left( \int_0^1 |\what{\BF^*}|^2 r \, dr \right)^{\frac12} \left( \int_0^1 \xi^2 |\hat{\psi}|^2 r \, dr + \int_0^1 \left| \frac{d}{dr}(r \hat{\psi} ) \right|^2 \frac1r \, dr   \right)^{\frac12}\\
&+ \frac{4\Phi}{\pi} \frac{\alpha}{4+ \alpha} |\xi| \int_0^1 \left| \frac{d}{dr}( r \hat{\psi} )  \right| | r \hat{\psi}| \, dr .
\ea
\ee
According to Lemma \ref{lemma1} and \eqref{3-33}, it holds that
\be \label{3-35} \ba
& \frac{4 \Phi}{\pi} \frac{\alpha}{4 + \alpha }|\xi| \int_0^1 \left| \frac{d}{dr} (  r \hat{\psi} ) \right| |r \hat{\psi}| \, dr \\
\leq & \frac14 \int_0^1 \left| \frac{d}{dr} ( r \hat{\psi} ) \right|^2 \frac1r \, dr + C \Phi^2 \xi^2 \int_0^1 |\hat{\psi}|^2 r \, dr \\
\leq & \frac14 \int_0^1 |\mathcal{L} \hat{\psi} |^2 r\, dr + \frac{C}{\epsilon_1} \left( \int_0^1 |\what{\BF^*}|^2 r \, dr  \right)^{\frac12}
\left[ \int_0^1 \left( \xi^2 |\hat{\psi}|^2 r + \left|   \frac{d}{dr}(r \hat{\psi}) \right|^2 \frac1r \right) \, dr   \right]^{\frac12} \\
\leq & \frac14 \int_0^1 |\mathcal{L} \hat{\psi} |^2 r\, dr + \frac14 \xi^2 \int_0^1 |\hat{\psi}|^2 r \, dr + \frac14 \int_0^1 \left| \frac{d}{dr} (r \hat{\psi}) \right|^2 \frac1r \, dr + C (\epsilon_1) \int_0^1 |\what{\BF^*}|^2 r \, dr \\
\leq & \frac12 \int_0^1 |\mathcal{L} \hat{\psi}|^2 r \, dr + \frac14 \xi^2 \int_0^1 \left| \frac{d}{dr} ( r \hat{\psi})  \right|^2 \frac1r \, dr +
C (\epsilon_1) \int_0^1 |\what{\BF^*}|^2 r \, dr .
\ea
\ee
Substituting \eqref{3-35} into \eqref{3-34} and using Young inequality give \eqref{3-31}. Thus  the proof of Proposition \ref{Bpropcase1} is completed.
\end{proof}

Let
\be \nonumber
\chi_1 (\xi ) = \left\{ \ba & 1 , \ \ \ |\xi| \leq \frac{1}{\epsilon_1 \Phi} , \\
& 0,\ \ \ \ \text{otherwise},  \ea  \right.
\ee
and $\psi_{low}$ be the function  such that
$
\what{\psi_{low}}  =  \chi_1 (\xi) \hat{\psi}.
$
Define
\be \nonumber
v^r_{low} = \partial_z \psi_{low},\ \ \ \ v^z_{low} = - \frac{\partial_r ( r \psi_{low} )}{r}, \ \ \ \mbox{and}\  \ \Bv^*_{low} = v^r_{low}\Be_r + v^z_{low}\Be_z.
\ee
Similarly, one can define $F^r_{low}, F^z_{low}, \BF^*_{low}$, and $\Bo^\theta_{low}$.

\begin{pro}\label{Bpropcase1-1} The solution $\Bv^*$  satisfies
\be \label{3-36}
\|\Bv^*_{low} \|_{H^2(\Omega)} \leq C \|\BF^*_{low}\|_{L^2(\Omega)},
\ee
and
\be \label{3-36-1}
\|v^r_{low} \|_{L^2(\Omega)} \leq C \Phi^{-1} \|\BF^*_{low} \|_{L^2(\Omega)} ,\ \ \ \ \|\partial_z v^z_{low} \|_{L^2(\Omega)}
\leq C \Phi^{-1} \|\BF^*_{low} \|_{L^2(\Omega)},
\ee
where $C$ is a uniform constant independent of $\Phi$, $\alpha$, and $\BF$.
\end{pro}

\begin{proof} Note that $\Bv^*_{low}$ is a strong solution to the following Stokes equations
\be \label{stokes-low} \left\{ \ba
& - \Delta \Bv^*_{low} + \nabla P = \BF^*_{low} - \bar{U} \partial_z \Bv^*_{low} - v^r_{low} \partial_r \bBU \ \ \ \mbox{in}\ \Omega, \\
& {\rm div}~\Bv^*_{low} = 0\ \ \ \ \mbox{in}\ \Omega.
\ea \right. \ee
According to the regularity theory for Stokes equations (\hspace{1sp}\cite[Lemma VI.1.2]{Galdi}) and the trace theorem for axisymmetric functions, one has
\be \label{3-37} \ba
\|\Bv^*_{low} \|_{H^2(\Omega)}
\leq & C \Big(\|\BF^*_{low}\|_{L^2(\Omega)} +  \Phi \|\partial_z \Bv^*_{low} \|_{L^2(\Omega)} +
 \Phi\| v^r_{low} \|_{L^2(\Omega)}  \\
&\ \ \ \ +  \|\Bv^*_{low}\|_{H^1(\Omega)} +  \|\partial_z \Bv^*_{low}\|_{H^1(\Omega)}\Big).
\ea
\ee
Herein, by virtue of the estimate \eqref{3-31} and Lemma \ref{lemma1}, it holds that
\be \label{3-38} \ba
\| v^r_{low} \|_{L^2(\Omega)} & \leq C \left( \int_{|\xi| \leq \frac{1}{\epsilon_1 \Phi} }\xi^2 \int_0^1 |\hat{\psi}|^2 r\, dr d\xi       \right)^{\frac12} \\
& \leq C \left(  \int_{|\xi| \leq \frac{1}{\epsilon_1 \Phi} } \xi^2 \int_0^1 |\mathcal{L} \hat{\psi} |^2 r \, dr d\xi         \right)^{\frac12}\\
& \leq C \left( \int_{|\xi| \leq \frac{1}{\epsilon_1 \Phi} }\frac{1}{(\epsilon_1\Phi)^2} C(\epsilon_1)  \int_0^1 \left| \what{\BF^*}  \right|^2 r \, dr d\xi       \right)^{\frac12} \\
& \leq C(\epsilon_1) \Phi^{-1} \|\BF^*_{low}\|_{L^2(\Omega)}
\ea
\ee
and
\be \label{3-39} \ba
& \| \partial_z \Bv^*_{low}\|_{L^2(\Omega)}  \leq \|\partial_z v^r_{low}\|_{L^2(\Omega)} + \|\partial_z v^z_{low} \|_{L^2(\Omega)} \\
\leq & C \left(  \int_{|\xi| \leq \frac{1}{\epsilon_1 \Phi} } \xi^4 \int_0^1  |\hat{\psi}|^2 r \, dr d\xi \right)^{\frac12} + C  \left(
\int_{|\xi| \leq \frac{1}{\epsilon_1 \Phi} } \xi^2 \int_0^1 \left| \frac{\partial}{\partial r} ( r \hat{\psi} ) \right|^2 \frac1r    \, dr d\xi         \right)^{\frac12} \\
\leq & C \Phi^{-1} \|\BF^*_{low}\|_{L^2(\Omega)} .
\ea \ee
Moreover,  one has
\be \label{3-30-0}
\|v^z_{low} \|_{L^2(\Omega)} \leq C \left(  \int_{ |\xi| \leq \frac{1}{\epsilon_1 \Phi} } \int_0^1 \left|  \frac{\partial }{\partial r}(r \hat{\psi} ) \right|^2 \frac1r \, dr d\xi    \right)^{\frac12}
\leq C \| \BF^*_{low} \|_{L^2(\Omega)}
\ee
and
\be \label{3-30} \ba
 \|  \Bv^*_{low} \|_{H^1 (\Omega)}   \leq  & \|\Bo^\theta_{low}\|_{L^2(\Omega)}  + \| \Bv^*_{low} \|_{L^2(\Omega)} \\
\leq & C \left( \int_{|\xi| \leq \frac{1}{\epsilon_1 \Phi} }  \int_0^1  | ( \mathcal{L} - \xi^2) \hat{\psi}|^2 r \, dr d\xi \right)^{\frac12}
+ C \|\BF^*_{low} \|_{L^2(\Omega)}  \\
\leq &   C \| \BF^*_{low} \|_{L^2(\Omega)}.
\ea \ee
Similarly, it holds that
\be \label{3-30-1}
\|\partial_z  \Bv^*_{low} \|_{H^1 (\Omega)} \leq  \|\partial_z \Bo^\theta_{low} \|_{L^2(\Omega)} +  \|\partial_z \Bv^*_{low} \|_{L^2(\Omega)}
\leq C \|\BF^*_{low} \|_{L^2(\Omega)} .
\ee
Summing \eqref{3-37}-\eqref{3-30-1} together gives \eqref{3-36}. This completes the proof of the proposition.
\end{proof}

\subsection{Uniform estimate for the case with large flux and high frequency}
In this subsection,  the uniform estimate for the solutions of \eqref{2-0-8}--\eqref{FBC} is established when the flux is large and the frequency is high.
\begin{pro}\label{Bpropcase2}
Assume that $|\xi| \geq \epsilon_1 \sqrt{\Phi} \geq 1 $. Let $\hat{\psi}$ be a  smooth solution to the problem \eqref{2-0-8}--\eqref{FBC}, then one has
\be \label{highf1}
 \int_0^1 |\mathcal{L} \hat{\psi}|^2 r \, dr + \xi^2 \int_0^1 \left| \frac{d}{dr} ( r \hat{\psi})    \right|^2 \frac{1}{r} \, dr
+ \xi^4 \int_0^1 |\hat{\psi}|^2 r \, dr
\leq C(\epsilon_1) |\xi|^{-2}  \int_0^1 |\what{\BF^*}|^2 r \, dr
\ee
and
\be \label{highf2}
 |\xi| \int_0^1 \frac{\bar{U}(r)}{r}  \left| \frac{d}{dr} ( r \hat{\psi}) \right|^2 \, dr +  |\xi|^3 \int_0^1 \bar{U}(r) |\hat{\psi}|^2 r \, dr
\leq C(\epsilon_1)  |\xi|^{-2}  \int_0^1 |\what{\BF^*}|^2 r \, dr .
\ee
\end{pro}

\begin{proof}  According to \eqref{3-7}-\eqref{3-7-1} and Lemma \ref{weightinequality}, one has
\be \label{3-41} \ba
& \frac{4 \Phi}{\pi} |\xi| \int_0^1 \left| \frac{d}{dr}(r \hat{\psi} )  \right| |r \hat{\psi} | \, dr \\
\leq  & C \Phi |\xi| \left( \int_0^1 |\mathcal{L} \hat{\psi} |^2 r \, dr    \right)^{\frac16} \left( \int_0^1 \frac{1-r^2}{r} \left|
  \frac{d}{dr}( r \hat{\psi} ) \right|^2 \, dr       \right)^{\frac13} \\
& \ \ \ \ \ \ \ \ \left( \int_0^1 \left| \frac{d}{dr}(r \hat{\psi} ) \right|^2 \frac1r \, dr     \right)^{\frac16} \left(  \int_0^1 (1  - r^2) |\hat{\psi} |^2 r \, dr         \right)^{\frac13} \\
\leq & \frac12 \int_0^1 |\mathcal{L} \hat{\psi} |^2 r\, dr + \xi^2 \int_0^1 \left| \frac{d}{dr}(r \hat{\psi})     \right|^2 \frac1r \, dr
+ C \Phi^{\frac32} \int_0^1 \frac{1-r^2}{r} \left|\frac{d}{dr} (r \hat{\psi})  \right|^2 \, dr \\
&\ \  \ + C \Phi^{\frac32} \xi^2 \int_0^1 (1  - r^2) |\hat{\psi} |^2 r \, dr \\
\leq & \frac12 \int_0^1 |\mathcal{L} \hat{\psi} |^2 r\, dr + \xi^2 \int_0^1 \left| \frac{d}{dr}(r \hat{\psi})     \right|^2 \frac1r \, dr
+ C \Phi^{\frac12} |\xi|^{-1} \left| \int_0^1 \hat{f} \overline{\hat{\psi} } r \, dr    \right| .
\ea
\ee
Substituting \eqref{3-41} into \eqref{3-9} gives
\be \label{3-42} \ba
& \int_0^1 | \mathcal{L} \hat{\psi} |^2 r \, dr + 2 \xi^2 \int_0^1 \left|\frac{d}{dr}(r \hat{\psi})  \right|^2 \frac1r \, dr + \xi^4 \int_0^1 |\hat{\psi}|^2 r \, dr \\
\leq & C \int_0^1 |\widehat{F^r} | |\xi \hat{\psi} | r \, dr + C \int_0^1 | \widehat{F^z} | \left| \frac{d}{dr}(r \hat{\psi}) \right| \, dr .
\ea
\ee
This, together with Young's inequality, gives \eqref{highf1}.

Taking \eqref{highf1} into \eqref{3-7} gives \eqref{highf2} and hence the proof of Proposition \ref{Bpropcase2} is completed.
\end{proof}

Let
\be \nonumber
\chi_2 (\xi ) = \left\{ \ba & 1 , \ \ \ |\xi| \geq \epsilon_1 \sqrt{\Phi}  , \\
& 0,\ \ \ \ \text{otherwise},  \ea  \right.
\ee
and $\psi_{high}$ be the function  such that
$
\what{\psi_{high}}  =  \chi_2 (\xi) \hat{\psi}.
$
Define
\be \nonumber
v^r_{high} = \partial_z \psi_{high},\ \ \ \ v^z_{high} = - \frac{\partial_r ( r \psi_{high} )}{r}, \ \ \ \  \text{and}\ \ \ \ \Bv^*_{high} = v^r_{high}\Be_r + v^z_{high}\Be_z.
\ee
Similarly, one can define $F^r_{high}, F^z_{high}, \BF^*_{high}$, and $\Bo_{high}^\theta $.

\begin{pro}\label{Bpropcase2-1} The solution $\Bv^*$  satisfies
\be \label{3-45}
\|\Bv^*_{high} \|_{H^{\frac53} (\Omega) }
\leq C (\epsilon_1) \|\BF^*_{high} \|_{L^2(\Omega) } ,\ \ \ \ \|\Bv^*_{high} \|_{H^2(\Omega)} \leq C(\epsilon_1)  \Phi^{\frac14}  \|\BF^*_{high} \|_{L^2(\Omega)},
\ee
and
\be \label{3-46}
\|v^r_{high} \|_{L^2(\Omega)} \leq C(\epsilon_1) \Phi^{-1} \|\BF^*_{high}\|_{L^2(\Omega)}, \ \ \ \ \|\partial_z v^z_{high} \|_{L^2(\Omega)}
\leq C(\epsilon_1) \Phi^{-\frac12} \|\BF^*_{high}\|_{L^2(\Omega)},
\ee
where $C$ is a uniform constant independent of $\Phi$, $\alpha$, and $\BF$.
\end{pro}

\begin{proof}
 $\Bv^*_{high}$ is a strong solution to the following Stokes equations,
\be \label{stokes-high} \left\{  \ba
& - \Delta \Bv^*_{high} + \nabla P = \BF^*_{high} - \bar{U} \partial_z \Bv^*_{high} - v^r_{high} \partial_r \bBU \ \ \ \mbox{in}\ \Omega, \\
& {\rm div}~\Bv^*_{high} = 0\ \ \ \ \mbox{in}\ \Omega.
\ea \right.
\ee
According to the regularity theory for Stokes equations (\hspace{1sp}\cite[Lemma VI.1.2]{Galdi}), one has
\be \label{3-47} \ba
\|\Bv^*_{high}\|_{H^2(\Omega)}
\leq &  C\Big( \|\BF^*_{high}\|_{L^2(\Omega)}
+   \| \bar{U}(r)  \partial_z \Bv^*_{high} \|_{L^2(\Omega)} +  \Phi \|v^r_{high}\|_{L^2(\Omega)}\\
&\ \ \ \  +  \|\Bv^*_{high}\|_{H^1(\Omega)} +  \|\partial_z \Bv^*_{high} \|_{H^1(\Omega)}\Big).
\ea
\ee
It follows from \eqref{highf1} and \eqref{highf2} that
\be \label{3-48} \ba
&  \| \bar{U}(r)  \partial_z \Bv^*_{high}\|_{L^2(\Omega)}  \\
\leq & C \left\{ \int_{|\xi| \geq \epsilon_1 \sqrt{\Phi}  }  \Phi  \left[ \xi^4 \int_0^1 \bar{U}(r) |\hat{\psi}|^2 r \, dr + \xi^2 \int_0^1 \frac{\bar{U}(r)}{r} \left|  \frac{\partial }{\partial r} ( r \hat{\psi})  \right|^2 \, dr   \right]  \, d\xi   \right\}^{\frac12} \\
\leq & C \left( \int_{|\xi| \geq \epsilon_1 \sqrt{\Phi}   } \Phi |\xi|^{-1} \int_0^1 |\hat{\BF}|^2 r \, dr d\xi  \right)^{\frac12}\\
\leq & C \Phi^{\frac14} \|\BF^*_{high}\|_{L^2(\Omega)},
\ea
\ee
\be \label{3-49} \ba
\Phi \|v^r_{high}\|_{L^2(\Omega)}&  \leq C \Phi \left( \int_{|\xi| \geq \epsilon_1 \sqrt{\Phi} } \xi^2 \int_0^1 |\hat{\psi} |^2 r \, dr d\xi   \right)^{\frac12} \leq C   \|\BF^*_{high}\|_{L^2(\Omega)},
\ea
\ee
and
\be \label{3-50}
\| v^z_{high}  \|_{ L^2 (\Omega)} \leq C \left(  \int_{ |\xi| \geq \epsilon_1 \sqrt{\Phi} }   \int_0^1 \left| \frac{\partial }{\partial r} ( r\hat{\psi})  \right|^2 \frac1r \, dr d\xi  \right)^{\frac12}  \leq C (\epsilon_1) \Phi^{-1}  \| \BF^*_{high}\|_{L^2(\Omega)}.
\ee
Furthermore, it holds that
\be \label{3-50-1} \ba
\|  \Bv^*_{high} \|_{H^1 (\Omega) } & \leq  \| \Bo^\theta_{high}\|_{L^2(\Omega)}  +  \|\Bv^*_{high} \|_{L^2(\Omega)} \\
  &  \leq C \left( \int_{|\xi| \geq \epsilon_1 \sqrt{\Phi} } \int_0^1 | (\mathcal{L} - \xi^2) \hat{\psi} |^2 r \, dr d\xi         \right)^{\frac12} + C \Phi^{-1} \|\BF^*_{high} \|_{L^2(\Omega)}   \\
& \leq C(\epsilon_1) \Phi^{-\frac12} \|\BF^*_{high} \|_{L^2(\Omega)}
\ea
\ee
and
\be \label{3-50-2} \ba
\|  \partial_z \Bv^*_{high} \|_{H^1 (\Omega) } & \leq  \| \partial_z \Bo^\theta_{high}\|_{L^2(\Omega)}
  + \|\partial_z \Bv^*_{high} \|_{L^2(\Omega)} \\
  & \leq C \left( \int_{|\xi| \geq \epsilon_1 \sqrt{\Phi} } \int_0^1 \xi^2  | (\mathcal{L} - \xi^2) \hat{\psi} |^2 r \, dr d\xi
         \right)^{\frac12}  + C \Phi^{-\frac12} \|\BF^*_{high} \|_{L^2(\Omega)} \\
& \leq C(\epsilon_1) \|\BF^*_{high} \|_{L^2(\Omega)} .
\ea \ee
Hence substituting  \eqref{3-48}-\eqref{3-50-2} into \eqref{3-47} gives the second estimate in \eqref{3-45}. And the interpolation between \eqref{3-50-1}  and
the second estimate in \eqref{3-45} gives the first estimate in \eqref{3-45}. Note that \eqref{3-49} is exactly the first estimate in \eqref{3-46}. On the other hand,
 \be \nonumber
 \|\partial_z v^z_{high}\|_{L^2(\Omega)}^2
 \leq C \int_{|\xi| \geq \epsilon_1 \sqrt{\Phi} } \xi^2 \int_0^1 \left| \frac{\partial }{\partial r}(r \hat{\psi}) \right|^2 \frac1r \, dr d\xi \leq C(\epsilon_1) \Phi^{-1} \|\BF^*_{high}\|_{L^2(\Omega)}^2,
 \ee
 which gives the second estimate \eqref{3-46}. This finishes the proof of the proposition.
\end{proof}

\subsection{Uniform estimate for the case with large flux  and  intermediate frequency}\label{sec-intermediate}
In this subsection, we give the uniform estimate for the solutions of \eqref{2-0-8}-\eqref{FBC} with respect to $\Phi$ and $\alpha$,  when the flux is large and the frequency is intermediate. The key ideas of the proof is quite similar to that in \cite{WX-Navier} for the flows periodic in the axial direction where $\xi$ takes only the integers. The major new difficulties comparing with the periodic case appear when we deal with the case with large flux, intermediate frequency, and intermediate slip coefficients.

The first proposition is about the case with large flux, intermediate frequency, and small slip coefficient.
\begin{pro}\label{Bpropcase3} Assume that $\Phi \gg 1$.
There exist two small constant $\epsilon_1 \in (0, 1)$ and $\delta \in (0, 1)$,  such that as long as $\frac{1}{\epsilon_1 \Phi}
\leq |\xi | \leq \epsilon_1 \sqrt{\Phi} $, and $4 +  \alpha \leq \delta (\Phi |\xi |)^{\frac13}$,
the solution $\hat{\psi}  (r)$ to the problem \eqref{2-0-8}-\eqref{FBC} can be decomposed into four parts,
\be \label{case3-1}
\hat{\psi} (r) = \what{\psi_{s} } (r) +  b\left[ \chi \what{\psi_{BL}} (r) + \what{\psi_{e}} (r) \right]  + a I_1(|\xi |r) .
\ee
The properties of these four parts are summarized as follows.

 $(1)$\ $\what{\psi_{s}}  $ is a solution to the following problem
\be \label{slip}
\left\{ \ba   & i \xi \bar{U}(r) ( \mathcal{L} - \xi^2) \what{ \psi_{s } } - (\mathcal{L} - \xi^2)^2 \what{\psi_{s}}  = \hat{f} = - \frac{d}{dr} \what{F^z} + i \xi \what{F^r} , \\
& \what{\psi_{s}}  (0) = \what{\psi_{s} }  (1) = \mathcal{L} \what{\psi_{s}}  (0) = \mathcal{L} \what {\psi_{s}}   (1) = 0.
\ea      \right.
\ee
$\psi_s$ satisfies the estimates
\be \label{case3-3} \ba
& \int_0^1 | \what{\psi_{s }} |^2 r  +
 \left|\frac{d}{dr}(r \what {\psi_{s}} )  \right|^2  \frac{1}{r}
+ \xi^2  \left| \what{\psi_{s }} \right|^2 r \, dr
\leq  C (\Phi |\xi |)^{- 2} (4 + \alpha)^2  \int_0^1 | \what{\BF^* }  |^2 r \, dr ,
\ea
\ee
\be \label{case3-4} \ba
& \int_0^1 | \mathcal{L} \what{\psi_{s }} |^2 r  + \xi^2  \left|  \frac{d}{dr} ( r \what{\psi_{s}} )\right|^2 \frac1r  + \xi^4
 | \what{ \psi_{s} } |^2  r \, dr
 \leq  C ( \Phi |\xi|)^{-1} (4 + \alpha) \int_0^1 |\what{\BF^*}  |^2 r \, dr ,
 \ea
\ee
\be \label{case3-5} \ba
& \int_0^1 \left| \frac{d}{dr}( r \mathcal{L} \what{ \psi_{s} }  )\right|^2 \frac1r  +
\xi^2  |\mathcal{L} \what {\psi_{s} }  |^2 r
+ \xi^4  \left|  \frac{d}{dr}( r \what{\psi_{s}} )       \right|^2 \frac1r  + \xi^6 | \what{\psi_{s}} |^2 r \, dr
\leq   C \int_0^1 | \what{ \BF^* }  |^2 r\, dr ,
\ea
\ee
and
\begin{equation}\label{case3-5-1}
\left| \frac{d}{dr} ( r \what{ \psi_{s} } ) (1) \right| \leq C ( \Phi |\xi|)^{-\frac34} (4 + \alpha)^{\frac34} \left( \int_0^1 | \what{\BF^*} |^2 r \, dr\right)^{\frac12} .
\end{equation}

$(2)$\ $I_1(\rho)$ is the modified Bessel function of the first kind, i.e.,
\be \label{eqBessel1}
\left\{
\ba & \rho^2 \frac{d^2}{d\rho^2} I_1 (\rho ) + \rho \frac{d}{d\rho} I_1 (\rho)  - (\rho^2 + 1) I_1 (\rho ) = 0, \\
& I_1 (0) = 0,\quad I_1(\rho ) >0 \,\,\text{if}\,\, \rho >0.
\ea
\right.
\ee
Furthermore, $a$ is a constant satisfying
\be \label{case3-6}
|a| \leq C (\Phi |\xi |)^{-\frac74} (4+  \alpha)^{\frac{11}{4}}   I_1( |\xi | )^{-1} \left( \int_0^1 | \what{\BF^*} |^2 r \, dr      \right)^{\frac12}.
\ee

$(3)$ \ $ \what{\psi_{BL }}$ is the boundary layer profile,
\be \label{case3-7}
\what{ \psi_{BL} }  (r) = e^{- \sqrt{\beta} \cos \frac{\theta}{2} (1 - r) } e^{-i \sqrt{\beta} \sin \frac{\theta}{2} (1 - r) },
\ee
where $\beta\in (0, +\infty)$ and $\theta$ are defined as follows
\be \label{defbeta}
\beta^2 = \left( \frac{\Phi \xi }{\pi} \right)^2 \left( \frac{4}{4+ \alpha} \right)^2 + \xi^4,\ \ \ \cos \theta = \frac{\xi^2}{\beta}, \ \ \text{and}\ \
\sin \theta = \frac{4\Phi \xi }{\pi (4 + \alpha) \beta}.
\ee
Moreover,  $\chi$ is a  smooth increasing cut-off function satisfying
\be \label{defchi}
\chi (r) = \left\{ \ba  &  1, \ \ \ \ r\geq \frac12,  \\ & 0, \ \ \  \ r \leq \frac14,   \ea  \right.
\ee
and the constant $b$ satisfies
\be \label{case3-8}
|b|\leq C (\Phi |\xi|)^{-\frac74} (4 +  \alpha)^{\frac{11}{4}}   \left( \int_0^1 |\what{ \BF^*} |^2 r \, dr  \right)^{\frac12}.
\ee

$(4)$\ $\what{ \psi_{e} } $ is a remainder term satisfying
\be \label{case3-9}
\left\{  \ba  &  i \xi  \bar{U}(r) ( \mathcal{L} - \xi^2) (\chi \what{\psi_{BL}}  + \what{ \psi_{e} } ) - (\mathcal{L} - \xi^2)^2 (\chi \what{\psi_{BL}} + \what{\psi_{e}} ) = 0,   \\
& \what{ \psi_{e} } (0) = \what{\psi_{e}} (1) = \mathcal{L} \what{\psi_{e}} (0) = \mathcal{L} \what{\psi_{e}} (1) = 0.
\ea  \right.
\ee
 And $ \what{ \psi_{e}}  $ satisfies the following estimates,
\be \label{case3-10}
\int_0^1 \left| \frac{d}{dr}(r \what{\psi_{e}} ) \right|^2 \frac1r  \, dr + \xi^2 \int_0^1 | \what{\psi_{e}} |^2 r \, dr  \leq C (4 +  \alpha)^2 \beta^{- \frac12} ,
\ee
%\be \label{case3-11}
%\int_0^1 |\mathcal{L} \what{ \psi_{e}}  |^2 r \, dr + \xi^2 \int_0^1 \left|\frac{d}{dr}(r \what{\psi_{e} }) \right|^2 \frac1r \, dr + \xi^4 \int_0^1 |\what{ \psi_{e} } |^2 r\, dr
%\leq C \Phi |\xi| (4 +  \alpha) \beta^{-\frac12},
%\ee
and
\be \label{case3-12} \ba
& \int_0^1 \left| \frac{d}{dr}( r \mathcal{L} \what{\psi_{e}} ) \right|^2 \frac1r  + \xi^2   |\mathcal{L} \what{\psi_{e}} |^2 r +
\xi^4 \left| \frac{d}{dr}( r \what{\psi_{e}} )  \right|^2 \frac1r    + \xi^6  | \what{\psi_{e}} |^2 r\, dr
\leq   C (\Phi |\xi|)^2 \beta^{-\frac12}.
\ea
\ee

In conclusion, $\hat{\psi} $ satisfies
\be \label{case3-13}
\int_0^1 \left| \frac{d}{dr}(r \hat{\psi} )\right|^2 \frac1r \, dr + \xi^2 \int_0^1 | \hat{\psi} |^2  r \, dr
\leq C (\Phi |\xi |)^{-\frac43} \int_0^1 | \what{\BF^*}|^2 r \, dr,
\ee
%\be \label{case3-14}
%\int_0^1 |\mathcal{L} \hat{\psi}|^2 r + \xi^2  \left| \frac{d}{dr}(r \hat{\psi} )\right|^2 \frac1r +  \xi^4 | \hat{\psi} |^2  r \, dr
%\leq C (\Phi |\xi|)^{-\frac23} \int_0^1 | \what{\BF^*}|^2 r \, dr,
%\ee
and
\be \label{case3-15}
\ba
&\int_0^1 \left| \frac{d}{dr}( r \mathcal{L} \hat{\psi} ) \right|^2 \frac1r  + \xi^2   |\mathcal{L} \hat{\psi} |^2 r +
\xi^4  \left| \frac{d}{dr}( r  \hat{\psi} )  \right|^2 \frac1r  + \xi^6 | \hat{\psi} |^2 r\, dr
\leq  C \int_0^1 | \what{\BF^*} |^2  r \, dr .
\ea
\ee
\end{pro}

\begin{proof}
	The proof of the proposition is almost same to that for \cite[Proposition 3.6]{WX-Navier}, so we give only the sketch here and refer to \cite{WX-Navier} for the details.
	
	 {\bf Step 1.}\ Find a solution $\what{\psi_s}$ to the problem \eqref{slip}. $\what{\psi_s}$ shares the same equation as that for $\hat{\psi}$. The boundary condition is different. In fact, one can get some good estimates for $\what{\psi_s}$ by classical energy estimate, due to the slip boundary condition.

{\bf Step 2.}\ We construct a boundary layer profile $\what{\psi_{BL}}$, which is defined by \eqref{case3-7} and satisfies the following equation
\be \nonumber
\left( i\frac{4}{4 + \alpha} \frac{\Phi \xi}{\pi}  - \frac{d^2}{dr^2} + \xi^2 \right) \left( \frac{d^2}{dr^2} - \xi^2  \right) \what{\psi_{BL}} = 0.
\ee
{Here $\frac{4}{4+ \alpha}\frac{\Phi}{\pi} $ in the first term is an approximation of $\bar{U}(r)$ near the boundary $r=1$. }

{\bf Step 3.}\ Since
\be \nonumber
i\xi \bar{U} (r) (\mathcal{L} - \xi^2) \what{\psi_{BL}} - (\mathcal{L} -\xi^2)^2 \what{\psi_{BL}} \neq 0,
\ee
we construct a remainder term $\what{\psi_e}$, which satisfies \eqref{case3-9}. Now $\what{\psi_s} + b(\chi \what{\psi_{BL}} + \what{\psi_e} )$ is a solution to the equation \eqref{2-0-8}.

{\bf Step 4.}\ $\what{\psi_s} + b(\chi \what{\psi_{BL}} + \what{\psi_e} )$ does not satisfy the boundary condition \eqref{FBC} yet. So we introduce the modified Bessel function $I_1(\rho)$, which is defined by \eqref{eqBessel1}. Then $I_1(|\xi| r)$ satisfies that $(\mathcal{L} -\xi^2) I_1 (|\xi| r) = 0$ and hence $I_1(|\xi| r)$ satisfies \eqref{2-0-8}. Choose some constants $a$ and $b$, such that $\hat{\psi} = \what{\psi_s} + b(\chi \what{\psi_{BL}} + \what{\psi_e} )+ aI_1(|\xi| r)$ is the solution to \eqref{2-0-8}-\eqref{FBC}.   At last, we do some estimates for the constants $a$, $b$, and the solution $\hat{\psi}$.
\end{proof}

%%%%%%%%%%%%%%%%%%%%%%%%%%%%%%%alpha is large%%%%%%%%%%%%%%%%%%%%%%%%%%%%%%%%%%%%
The next proposition is about the  case with large flux, intermediate frequency, and large slip coefficient.
\begin{pro} \label{case4}
Assume that $\Phi \gg 1$. There exist two small independent positive constants $\epsilon_1$ and  $\delta$, such that as long as
$\frac{1}{\epsilon_1 \Phi} \leq |\xi | \leq \epsilon_1 \sqrt{\Phi}$ and $4 + \alpha\geq \delta^{-1} (\Phi |\xi |)^{\frac13}>5$, the solution $\hat{\psi} (r)$ to the problem \eqref{2-0-8}-\eqref{FBC} can be decomposed into four parts,
\be \nonumber
\what{\psi} (r) = \what{ \psi_{s}} (r) + b \left[ \chi \what{\psi_{BL}}(r) + \what{\psi_{e}} (r)    \right] + a I_1 (|\xi | r).
\ee
Here $(1)$\ $ \what{\psi_{s} }$ is a solution to the problem \eqref{slip} satisfying
\be \label{propslip1-large-6}
\int_0^1 \left|\frac{d}{dr}(r  \what{\psi_{s} } ) \right|^2  \frac{1}{r}\, dr
+ \xi^2 \int_0^1 \left| \what{ \psi_{s} }  \right|^2 r \, dr
\leq C (\Phi |\xi |)^{- \frac43 } \int_0^1 | \what{ \BF^* }  |^2 r \, dr ,
\ee
\be \label{propslip1-large-7}
\begin{aligned}
&\int_0^1 | \mathcal{L} \what{ \psi_{s} } |^2 r \, dr + \xi^2 \int_0^1 \left|  \frac{d}{dr} ( r \what {\psi_{s} }  )\right|^2 \frac1r \, dr + \xi^4
\int_0^1 | \what{ \psi_{s} } |^2  r \, dr\\
 \leq & C ( \Phi |\xi |)^{-\frac23} \int_0^1 | \what{\BF^* }|^2 r \, dr ,
\end{aligned}
\ee
\be \label{propslip1-large-8} \ba
& \int_0^1 \left| \frac{d}{dr}( r \mathcal{L} \what{\psi_{s} }  )\right|^2 \frac1r  +
\xi^2  |\mathcal{L} \what{ \psi_{s} }  |^2 r  + \xi^4  \left|  \frac{d}{dr}( r \what{\psi_{s}} )       \right|^2 \frac1r    + \xi^6  | \what{\psi_{s} }  |^2 r \, dr
\leq   C \int_0^1 | \what{\BF^*} |^2 r\, dr ,
\ea
\ee
and
\be \label{propslip1-large-9}
\left| \frac{d}{dr} \what{\psi_{s} }  (1) \right| \leq C (\Phi |\xi |)^{-\frac12} \left( \int_0^1 |\what{ \BF^* } |^2 r \, dr \right)^{\frac12}.
\ee

$(2)$\ $I_1(\rho)$ is the modified Bessel function of the first kind as in Proposition \ref{Bpropcase3}, and $a$ is a constant satisfying
\be \nonumber
|a| \leq C (\Phi |\xi |)^{-\frac56} I_1(|\xi |)^{-1} \left(   \int_0^1 | \what{\BF^* }|^2 r \, dr  \right)^{\frac12}.
\ee

$(3)$\ $\what{\psi_{BL}}$ is the boundary layer profile,
\be \label{case4-7}
\what{ \psi_{BL}} (r)  = G_{\xi, \Phi} ( |\beta| (1 - r)) \quad \text{with}\quad |\beta|=\left(\frac{4\Phi |\xi | }{\pi}\right)^{\frac13},
\ee
where $G_{\xi,  \Phi}(\rho) $ is a smooth function which decays exponentially at infinity and is uniformly bounded in the set
\be \nonumber
\mcE = \left\{ ( \xi, \Phi, \rho):\ \Phi \geq 1, \ \frac{1}{\Phi} \leq  |\xi | \leq \sqrt{\Phi},\  0 \leq \rho < + \infty   \right\}.
\ee
Moreover, $\chi$ is the smooth function satisfying \eqref{defchi} and $b$ is a constant satisfying
\be \nonumber
|b| \leq C (\Phi |\xi |)^{-\frac56}  \left(   \int_0^1 |\what{ \BF^* } |^2 r \, dr  \right)^{\frac12}.
\ee

$(4)$\ $ \what{ \psi_{e} } $ is a remainder term satisfying \eqref{case3-9} and the following estimate
\be \label{case4-10}
\int_0^1 \left| \frac{d}{dr}(r \what{ \psi_{e} } ) \right|^2 \frac1r \, dr + \xi^2 \int_0^1 | \what{\psi_{e}} |^2 r \, dr \leq C (\Phi |\xi |)^{\frac13} ,
\ee
%\be \label{case4-11}
%\int_0^1 |\mathcal{L} \what{ \psi_{e} } |^2 r \, dr + \xi^2 \int_0^1 \left| \frac{d}{dr}(r \what{ \psi_{e} } ) \right|^2 \frac1r \, dr + \xi^4 \int_0^1 | \what{\psi_{e}} |^2 r \, dr \leq C \Phi |\xi|,
%\ee
and
\be \label{case4-12} \ba
& \int_0^1 \left| \frac{d}{dr} ( r \mathcal{L} \what{ \psi_{e} } ) \right|^2 \frac1r  + \xi^2  |\mathcal{L} \what{\psi_{e} } |^2 r  + \xi^4  \left| \frac{d}{dr} (r  \what{ \psi_{e} } )  \right|^2 \frac1r
 + \xi^6 | \what{ \psi_{e} } |^2 r \, dr
\leq C (\Phi |\xi | )^{\frac53}.
\ea
\ee

In conclusion, $\what{ \psi } $ satisfies
\be \label{case4-13}
\int_0^1 \left|\frac{d}{dr}(r \what{ \psi}  ) \right|^2 \frac1r \, dr + \xi^2 \int_0^1 | \what{ \psi } |^2 r \, dr \leq C (\Phi |\xi| )^{-\frac43} \int_0^1 | \what{\BF^*} |^2 r \, dr,
\ee
%\be  \label{case4-14}
%\int_0^1 |\mathcal{L} \what{ \psi} |^2 r +
%\xi^2  \left| \frac{d}{dr}( r  \what{ \psi } )  \right|^2 \frac1r  + \xi^4  | \what{ \psi}  |^2 r\, dr
%\leq C (\Phi |\xi|)^{-\frac23} \int_0^1 | \what{\BF^*} |^2 r \, dr,
%\ee
and
\be \label{case4-15}
\ba
&\int_0^1 \left| \frac{d}{dr}( r \mathcal{L} \what{ \psi}  ) \right|^2 \frac1r  + \xi^2   |\mathcal{L} \what{ \psi} |^2 r +
\xi^4  \left| \frac{d}{dr}( r  \what{ \psi } )  \right|^2 \frac1r  + \xi^6  | \what{ \psi}  |^2 r\, dr
\leq & C \int_0^1 | \what{ \BF^* } |^2  r \, dr .
\ea
\ee
\end{pro}

The key idea for the proof of Proposition \ref{case4} is quite similar to that for Proposition \ref{Bpropcase3}. The difference lies in the choice of boundary layer profile $\what{\psi_{BL}}$. Here $\what{\psi_{BL}}$ is the solution to the approximate equation
\be \nonumber
\left( i \frac{\Phi \xi}{\pi} 4 (1 - r) - \frac{d^2}{dr^2} + \xi^2 \right) \left( \frac{d^2}{dr^2} - \xi^2 \right) \what{\psi_{BL}} = 0,
\ee
where $\frac{4 \Phi}{\pi} (1 - r)$ in the first term is an approximation of $\bar{U}(r)$ near the boundary $r=1$ when $\alpha$ is large. The detailed proof for Proposition \eqref{case4} is almost the same to that for \cite[Proposition 3.8]{WX-Navier}.

\vspace{5mm}
%%%%%%%%%%%%%%%%%%\alpha is medium %%%%%%%%%%%%%%%%%%%%%%%%%%%%
{Now we are in position to study the case with large flux, intermediate frequency, and intermediate slip coefficient. We can also use the methods developed in Propositions \ref{Bpropcase3} and \ref{case4} to estimate the solutions. However, it seems hard for us to get the same uniform estimates with respect to both $\Phi$ and $\alpha$. As a compromise, we try to prove a uniform estimate the solutions in a less regular space ($H^1$ for $v^r$ and $v^z$). This, together with some better estimate for $v^r$ (see Proposition \ref{medregularity2}), also helps to close the estimate for the nonlinear problem in the next section.
The proof for this case here is more complicated than the case for periodic flows\cite{WX-Navier}, where $|\xi|\geq 1$ as long as the frequency is not zero.}
\begin{pro}\label{case5}
Assume that $\Phi \gg 1$ and $\epsilon_1$, $\delta$ are two independent constants in $(0, 1)$. As long as $\xi$ satisfies
\begin{equation}\label{intermediate}
\frac{1}{\epsilon_1 \Phi} \leq |\xi | \leq \epsilon_1 \sqrt{\Phi}\quad  \text{and} \quad   \delta (\Phi |\xi|)^{\frac13} \leq 4 + \alpha \leq \delta^{-1} (\Phi |\xi|)^{\frac13},
\end{equation}
 the solution $\what{\psi}$ to \eqref{2-0-8} satisfies
\be \label{6-0}
 \int_0^1 |\what{ \psi } |^2 r \, dr \leq C (\Phi |\xi |)^{-\frac53} \int_0^1 |\what{\BF^*}|^2 r \, dr,
\ee
\be \label{6-1}
\int_0^1 \left| \frac{d}{dr}(r \what{ \psi} ) \right|^2 \frac1r \, dr + \xi^2 \int_0^1 |\what{ \psi} |^2 r \, dr \leq C (\Phi |\xi| )^{-\frac43} \int_0^1 |\what{ \BF^* } |^2 r \, dr ,
\ee
\be \label{6-2}
|\xi| \int_0^1  \bar{U}(r) \left| \frac{d}{dr} ( r \what{ \psi} )\right|^2 \frac1r \, dr + |\xi|^3 \int_0^1 \bar{U}(r) |\what{ \psi} |^2 r \, dr \leq C (\Phi |\xi|)^{-\frac23} \int_0^1 |\what{ \BF^*} |^2 r \, dr,
\ee
and
\be \label{6-3} \ba
& \int_0^1 |\mathcal{L} \what{ \psi} |^2 r  + \xi^2  \left|\frac{d}{dr} ( r \what{ \psi} )  \right|^2 \frac1r  + \xi^4  |\what{ \psi} |^2 r \, dr
+  \alpha \left| \frac{d}{dr} ( r \what{ \psi} ) (1) \right|^2 \\
\leq &
C (\Phi |\xi |)^{-\frac12} \int_0^1 |\what{ \BF^*} |^2 r \, dr.
\ea
\ee
\end{pro}
\begin{proof}
It follows from \eqref{intermediate} that one has
\be \nonumber
\bar{U}(r) = \frac{2\Phi}{\pi} (1 - r^2) \frac{\alpha}{4+ \alpha} + \frac{\Phi}{\pi} \frac{4}{4+ \alpha} \geq \frac{4\Phi}{\pi} \delta (\Phi |\xi | )^{-\frac13}.
\ee
This, together with \eqref{3-7}, yields
\be \nonumber \ba
& (\Phi |\xi | )^{\frac23} \left( \int_0^1 \left| \frac{d}{dr}(r \what{ \psi} ) \right|^2 \frac1r \, dr + \xi^2 \int_0^1 |\what{ \psi} |^2r \, dr \right) \\
\leq & C (\delta) \left[  \int_0^1 |\what{F^r}| |\xi \what{\psi} | r \, dr + \int_0^1 |\what{ F^z} | \left| \frac{d}{dr}(r \what{ \psi} )    \right| \, dr              \right].
\ea \ee
Applying Schwarz inequality gives \eqref{6-1}.
Taking \eqref{6-1} into \eqref{3-7} yields \eqref{6-2}.

Next, using the property $\bar{U}(r) \geq \frac{\Phi}{\pi} (1- r^2)$, Lemma \ref{lemmaHLP}, and \eqref{6-2} gives
\be \nonumber
\int_0^1 |\what{ \psi} |^2 r \, dr \leq
C\Phi^{-1}  \int_0^1 \frac{\bar{U}(r)}{r} \left| \frac{d}{dr}(r \what{ \psi} )  \right|^2 \, dr  \leq C(\delta) (\Phi |\xi| )^{-\frac53} \int_0^1 |\what{ \BF^*} |^2 r \, dr.
\ee
This is exactly the estimate \eqref{6-0}.

Finally, it follows from \eqref{3-5}, \eqref{6-0}, and  \eqref{6-1} that
\be \label{6-9} \ba
& \int_0^1 |\mathcal{L} \what{ \psi} |^2 r \, dr +  2\xi^2 \int_0^1 \left| \frac{d}{dr}(r \what{ \psi} ) \right|^2 \frac1r \, dr + \xi^4 \int_0^1 |\what{ \psi} |^2 r \, dr \\
\leq & \int_0^1 |\what{ F^r} | |\xi \what{ \psi} | r \, dr + \int_0^1 |\what{ F^z}| \left| \frac{d}{dr}(r \what{ \psi} )  \right| \, dr + \frac{4\Phi |\xi| }{\pi}
\int_0^1 \left| \frac{d}{dr} (r \what{ \psi} )  \right| | r \what{ \psi} | \, dr \\
\leq & C (\Phi |\xi| )^{-\frac23} \int_0^1 |\what{ \BF^*} |^2 r \, dr + C \Phi |\xi | \left( \int_0^1 \left| \frac{d}{dr}(r \what{ \psi} ) \right|^2 \frac1r \, dr \right)^{\frac12} \left(
\int_0^1 |\what{ \psi} |^2 r \, dr \right)^{\frac12}\\
\leq & C (\Phi |\xi|)^{-\frac12} \int_0^1 |\what{ \BF^*} |^2 r \, dr .
\\
\ea \ee
This completes the proof of the proposition.
\end{proof}

%%%%%%%%%%%%%%%%%%%%%%%%%%%%%%%%%%%%%%%%%%%%%%%%%%%%%%%%%%%%%%%%

Let
\be \nonumber
Z_1 = \left\{ \xi \in \mathbb{R}:\ \frac{1}{\epsilon_1 \Phi}  \leq |\xi| \leq \epsilon_1 \sqrt{\Phi}, \ \ |\xi| \geq \delta^{-3} (4 + \alpha)^3 \Phi^{-1}        \right\} ,
\ee
\be \nonumber
Z_2 = \left\{ \xi \in \mathbb{R}:\ \frac{1}{\epsilon_1 \Phi} \leq |\xi| \leq \epsilon_1 \sqrt{\Phi},\  \ |\xi| \leq \delta^3 (4 +  \alpha)^3 \Phi^{-1}         \right\},
\ee
\be \nonumber
Z_3 = \left\{ \xi \in \mathbb{R} :\  \frac{1}{\epsilon_1 \Phi}  \leq |\xi| \leq \epsilon_1 \sqrt{\Phi},\ \ \delta^3 (4 + \alpha)^3 \Phi^{-1} < |\xi| < \delta^{-3} (4  + \alpha)^3 \Phi^{-1} \right\}.
\ee
For $k= 1, \,  2,\  3$, let
\be \nonumber
\chi_{3, k}  (\xi ) = \left\{ \ba & 1 , \ \ \  \xi \in Z_k, \\
& 0,\ \ \ \ \text{otherwise}  \ea  \right.
\ee
and $\psi_{med, k}$ be the function  such that
$
\what{\psi_{med, k} }  =  \chi_{3, k}  (\xi) \what{\psi}.
$
Define
\be \nonumber
v^r_{med, k } = \partial_z \psi_{med, k },\ \ \ \ v^z_{med} = - \frac{\partial_r ( r \psi_{med, k } )}{r}, \ \ \ \ \text{and}\ \ \ \ \Bv^*_{med, k } = v^r_{med, k }\Be_r + v^z_{med, k }\Be_z.
\ee
Similarly, one can define $F^r_{med, k }, F^z_{med, k }, \BF^*_{med, k }$, and $\Bo^\theta_{med, k }$.

\begin{pro}\label{medregularity1}
The solution $\Bv^*$ satisfies
\be \label{7-1}
\|\Bv^*_{med, 1} \|_{H^2(\Omega)} \leq C \|\BF^*_{med, 1}\|_{L^2(\Omega)}, \ \ \ \ \  \  \|\Bv^*_{med, 2} \|_{H^2(\Omega)} \leq C \|\BF^*_{med, 2}\|_{L^2(\Omega)},
\ee
and
\be \label{7-1-1}
\|v^r_{med, 1} \|_{L^2(\Omega)} \leq C \Phi^{-\frac45} \|\BF^*_{med, 1} \|_{L^2(\Omega)} ,\ \ \ \ \|v^r_{med, 2} \|_{L^2(\Omega)} \leq C \Phi^{-\frac45} \|\BF^*_{med, 2} \|_{L^2(\Omega)},
\ee
\be \label{7-1-2}
\|\partial_z v^z_{med, 1} \|_{L^2(\Omega)} \leq C \Phi^{-\frac37} \|\BF^*_{med, 1} \|_{L^2(\Omega)},\ \ \ \ \|\partial_z v^z_{med, 2} \|_{L^2(\Omega)} \leq C \Phi^{-\frac37} \|\BF^*_{med, 2} \|_{L^2(\Omega)},
\ee
where the constant $C$ is a uniform constant independent of $\Phi$, $\alpha$, and $\BF$.
\end{pro}

\begin{proof}
It follows from Proposition \ref{Bpropcase3} that
\be \label{7-2}
\|\Bv_{med, 1}^* \|_{L^2(\Omega)}^2 \leq C \int_{\xi \in Z_1} \int_0^1 \left( \xi^2 |\what{ \psi} |^2 r + \left| \frac{\partial }{\partial r}(r \what{ \psi} ) \right|^2 \frac1r   \right) \, dr d\xi  \leq C \| \BF_{med, 1}^*\|_{L^2(\Omega)}^2.
\ee
Similarly, one has
\be \label{7-3-0} \ba
\|\Bv_{med, 1}^*\|_{H^1(\Omega)}^2 & = \|\omega^\theta_{med, 1} \|_{L^2(\Omega)}^2 +  \|\Bv^*_{med, 1 }\|_{L^2(\Omega)}^2 \\
& \leq C \int_{\xi \in Z_1}  \int_0^1 | (\mathcal{L} - \xi^2 ) \what{ \psi}   |^2 r \, dr  d\xi + C  \|\Bv^*_{med, 1 }\|_{L^2(\Omega)}^2\\
& \leq C \|\BF^*_{med, 1} \|_{L^2(\Omega)}^2
\ea \ee
and
\be \label{7-5-1} \ba
\|\partial_z \Bv^*_{med, 1} \|_{H^1(\Omega)}^2 &=  \|\partial_z \omega^\theta_{med, 1} \|_{L^2(\Omega)}^2  +  \|\Bv^*_{med, 1}\|_{H^1(\Omega)}^2 \\
& \leq C  \int_{\xi \in Z_1}  \int_0^1 \xi^2 | (\mathcal{L} - \xi^2 ) \what{ \psi}   |^2 r \, dr d\xi  + C  \|\BF^*_{med, 1 }\|_{L^2(\Omega)}^2\\
& \leq C \|\BF^*_{med, 1} \|_{L^2(\Omega)}^2.
\ea
\ee
It follows from the estimates \eqref{7-3-0}-\eqref{7-5-1} and the trace theorem for axisymmetric functions that
\be \label{7-6}
\|\Bv^*_{med, 1} \|_{H^{\frac32} (\partial \Omega)} \leq C \|\Bv^*_{med, 1} \|_{H^1(\Omega)} + C \|\partial_z \Bv^*_{med, 1}\|_{H^1(\Omega)}
\leq C \|\BF^*_{med, 1} \|_{L^2(\Omega)}.
\ee

On the other hand, the straightforward computations give
\be \label{7-9-0}
\Delta \Bv^*_{med, 1 } = -{\rm curl}~\Bo^\theta_{med, 1} = \partial_z \omega^\theta_{med, 1} \Be_r - \left( \partial_r \omega^\theta_{med, 1} + \frac{\omega^\theta_{med, 1}}{r} \right)\Be_z.
\ee
Thus according to Proposition \ref{Bpropcase3}, one has
\be \nonumber \ba
\left\| \partial_r \omega^\theta_{med, 1} + \frac{\omega^\theta_{med, 1}}{r}   \right\|_{L^2(\Omega)}^2
&= \left\| \frac1r \frac{\partial}{\partial r} (r \omega^\theta_{med, 1} )     \right\|_{L^2(\Omega)}^2 \\
& =  \int_{\xi \in Z_1}   \int_0^1 \left| \frac{\partial }{\partial  r} [ r ( \mathcal{L} - \xi^2) \what{ \psi}  ]   \right|^2
\frac1r \, dr d\xi   \\
& \leq C\|\BF^*_{med, 1 } \|_{L^2(\Omega)}^2 .
\ea
\ee
Applying the regularity theory for the elliptic equation \eqref{7-9-0} for $\Bv^*_{med, 1}$ yields
\be \label{7-15}
\| \Bv^*_{med, 1} \|_{H^2(\Omega)}
\leq C \|{\rm curl}~\Bo_{med, 1}^\theta \|_{L^2(\Omega)} + C \|\Bv^*_{med, 1} \|_{H^{\frac32}(\partial \Omega)} \leq C \|\BF^*_{med, 1} \|_{L^2(\Omega)}.
\ee

It follows from Proposition \ref{Bpropcase3} that
\be \label{7-16}
\int_0^1 \left| \frac{d}{dr} ( r \hat{\psi}) \right|^2 \frac1r \, dr + \xi^2 \int_0^1 |\hat{\psi}|^2 r \, dr \leq C (\Phi |\xi|)^{-\frac43} \int_0^1 |\what{\BF^*} |^2 r \, dr , \ \ \text{for any}\,\, \xi \in Z_1.
\ee
For any $\xi\in Z_1$,
combining the two estimates \eqref{7-16} and \eqref{3-7} yields
\be \label{7-18}
|\xi| \int_0^1 \frac{\bar{U}(r) }{r} \left|\frac{d}{dr} ( r \hat{\psi})   \right|^2 \, dr + |\xi|^3 \int_0^1 \bar{U}(r) |\hat{\psi}|^2 r \, dr
\leq C (\Phi |\xi|)^{-\frac23} |\what{\BF^*}|^2 r \, dr.
\ee
Since $\bar{U}(r) \geq \frac{\Phi}{\pi} (1 - r^2) $, one has
\be \label{7-19}
\int_0^1 |\hat{\psi}|^2 r \, dr \leq C\Phi^{-1} \int_0^1 \frac{\bar{U}(r)}{r} \left| \frac{d}{dr}(r \hat{\psi}) \right|^2 \, dr
\leq C (\Phi |\xi|)^{-\frac53} \int_0^1 |\what{\BF^*}|^2 r \, dr .
\ee
Interpolation between \eqref{7-16} and \eqref{7-19} gives
\be \nonumber
\xi^2 \int_0^1 |\hat{\psi}|^2 r \, dr \leq C \Phi^{-\frac85} \int_0^1 |\BF^*|^2 r \, dr,\ \ \ \xi \in Z_1.
\ee
This implies
\be \label{7-21}
\|v^r_{med, 1} \|_{L^2(\Omega)} \leq C \Phi^{-\frac45} \|\BF^*_{med, 1} \|_{L^2(\Omega)}.
\ee

Next, following almost the same lines as in the proof of \eqref{6-9}  gives
\be \label{7-22}
\int_0^1 |\mathcal{L} \hat{\psi}|^2 r \, dr + \xi^2 \int_0^1 \left| \frac{d}{dr}( r \hat{\psi} ) \right|^2 \frac1r \, dr
\leq C (\Phi |\xi|)^{-\frac12} \int_0^1 |\what{\BF^*} |^2 r \, dr ,\ \ \ \ \xi \in Z_1.
\ee
Interpolation between \eqref{7-16} and \eqref{7-22} gives that
\be \nonumber
\xi^2 \int_0^1 \left| \frac{d}{dr}( r \hat{\psi})  \right|^2 \frac1r \, dr \leq C \Phi^{-\frac67} \int_0^1 |\what{\BF^*}|^2 r \, dr,\ \ \ \ \xi \in Z_1.
\ee
This implies
\be \label{7-23}
\|\partial_z  v^z_{med, 1} \|_{L^2(\Omega)} \leq C \Phi^{-\frac37} \|\BF^*_{med, 1} \|_{L^2(\Omega)}.
\ee

Following the same lines as above and applying Proposition \ref{case4}, one can prove that the same estimates hold for $\Bv^*_{med, 2}$.
This finishes the proof of Proposition \ref{medregularity1}.
\end{proof}

\begin{pro}\label{medregularity2}
The solution $\Bv^*$ satisfies
\be \label{8-1}
\|\Bv^*_{med, 3} \|_{H^{1} (\Omega)} \leq C \|\BF^*_{med, 3} \|_{L^2(\Omega)}, \ \ \ \ \ \|\Bv^*_{med, 3} \|_{H^{2} (\Omega)} \leq C \Phi^{\frac14} \|\BF^*_{med, 3} \|_{L^2(\Omega)},
\ee
and
\be \label{8-1-1}
\|v^r_{med, 3} \|_{L^2(\Omega)} \leq C \Phi^{-\frac45} \|\BF^*_{med, 3} \|_{L^2(\Omega)},\ \ \ \ \ \|\partial_z v^z_{med, 3} \|_{L^2(\Omega)}
\leq C \Phi^{-\frac37} \|\BF^*_{med, 3} \|_{L^2(\Omega)},
\ee
where the constant $C$ is independent of $\Phi$, $\alpha$, and $\BF^*$.
\end{pro}
\begin{proof}
 $\Bv^*_{med, 3} $ satisfies the following equation
\be \label{8-2}
\left\{
\ba
& - \Delta \Bv^*_{med, 3} + \nabla P = \BF^*_{med, 3}- \bar{U} \partial_z \Bv^*_{med, 3}  - v^r_{med, 3} \partial_r \bBU\ \ \ \ \mbox{in}\ \Omega,\\
& {\rm div}~ \Bv^*_{med, 3} = 0\ \ \ \ \mbox{on}\ \partial \Omega.
\ea
\right.
\ee
It follows from the regularity theory for the Stokes equations and trace theorem for axisymmetric functions that
\be \label{8-3}
\ba
\|\Bv^*_{med, 3} \|_{H^2(\Omega)}
 \leq & C \Big(\|\BF^*_{med, 3} \|_{L^2(\Omega)} +  \|\bar{U} \partial_z \Bv^*_{med, 3} \|_{L^2(\Omega)} +  \|v^r_{med, 3}  \partial_r \bar{U} \|_{L^2(\Omega)}\\
 &\ \ \ + \|\Bv^*_{med, 3} \|_{H^1(\Omega)} +  \|\Bv^*_{med, 3}\|_{H^{\frac32}(\partial \Omega)}\Big) \\
 \leq &  C\Big( \|\BF^*_{med, 3} \|_{L^2(\Omega)} +  \|\bar{U} \partial_z \Bv^*_{med, 3} \|_{L^2(\Omega)} +  \|v^r_{med, 3}  \partial_r \bar{U} \|_{L^2(\Omega)}\\
 &\ \ \ +  \|\Bv^*_{med, 3} \|_{H^1(\Omega)} +  \|\partial_z \Bv^*_{med, 3}\|_{H^1 ( \Omega)}\Big).
\ea
\ee
Herein, according to Proposition \ref{case5}, one has
\be \label{8-5} \ba
& \|\bar{U} \partial_z \Bv^*_{med, 3} \|_{L^2(\Omega)}^2 \\
\leq & C \Phi  \int_{\xi  \in Z_3 } \left[\xi^2 \int_0^1 \bar{U}(r) \left| \frac{\partial }{\partial r}(r \what{ \psi} ) \right|^2 \frac1r \, dr   +
\xi^4 \int_0^1 \bar{U}(r) |\what{\psi} |^2 r \, dr   \right] d\xi   \\
\leq & C \Phi \int_{\xi \in Z_3} |\xi| (\Phi |\xi| )^{-\frac23} \int_0^1 |\what{ \BF^*} |^2 r \, dr d\xi \\
\leq & C \Phi^{\frac12} \|\BF^*_{med, 3} \|_{L^2(\Omega)}^2
\ea \ee
and
\be \label{8-6} \ba
& \| v^r_{med, 3} \partial_r \bar{U} \|_{L^2(\Omega)}^2
\leq  C \Phi^2 \int_{\xi \in Z_3} \xi^2 \int_0^1 |\what{ \psi} |^2 r \, dr d\xi  \\
\leq & C \int_{\xi \in Z_3} (\Phi |\xi| )^{\frac13} \int_0^1 |\what{\BF^*} |^2 r \, dr d\xi \\
\leq & C(\epsilon_1) \Phi^{\frac12} \|\BF^*_{med, 3}\|_{L^2(\Omega)}^2.
\ea \ee
Similarly, it holds that
\be  \label{8-7}
\ba
 \|\Bv^*_{med, 3} \|_{H^1(\Omega)}^2
=  & \|\Bo^\theta_{med, 3}\|_{L^2(\Omega)}^2 +  \|\Bv^*_{med, 3}\|_{L^2(\Omega)}^2\\
\leq & C \int_{\xi \in Z_3} \left[ \int_0^1 |(\mathcal{L} - \xi^2) \what{\psi}|^2 r \, dr +  \int_0^1 \left|\frac{\partial }{\partial r} (r \what{ \psi} )  \right|^2 \frac1r + \xi^2 |\what{ \psi} |^2 r \, dr  \right]  d\xi \\
\leq & C (\epsilon_1) \|\BF^*_{med, 3} \|_{L^2(\Omega)}^2
\ea
\ee
and
\be \label{8-8}
\ba
 \|\partial_z \Bv^*_{med, 3} \|_{H^1(\Omega)}^2
= &   \|\partial_z \Bo^\theta_{med, 3} \|_{L^2(\Omega)}^2  +  \|\partial_z \Bv^*_{med, 3} \|_{L^2(\Omega)}^2  \\
\leq & C \int_{\xi \in Z_3}  \int_0^1 \xi^2 |(\mathcal{L} - \xi^2) \what{ \psi} |^2 r \, dr d\xi   + C \|\Bv^*_{med, 3} \|_{H^1(\Omega)}^2 \\
\leq & C \int_{\xi \in Z_3} \xi^2 (\Phi |\xi|)^{-\frac12} \int_0^1  | \what{ \BF^*} |^2 r \, dr d\xi  + C \|\Bv^*_{med, 3} \|_{H^1(\Omega)}^2\\
\leq & C \Phi^{\frac14} \|\BF^*_{med, 3}\|_{L^2(\Omega)}^2.
\ea
\ee
Taking \eqref{8-5}--\eqref{8-8} into \eqref{8-3} gives \eqref{8-1}.

On the other hand, interpolation between \eqref{6-0} and \eqref{6-1} yields that for every $\xi \in Z_3$,
\be \nonumber
\xi^2 \int_0^1 |\what{\psi}|^2 r \, dr \leq C \Phi^{-\frac85} \int_0^1 |\what{\BF^*} |^2 r \, dr .
\ee
This implies
\be \label{8-11}
\| v^r_{med, 3} \|_{L^2(\Omega)} \leq C \Phi^{-\frac45} \|\BF^*_{med, 3} \|_{L^2(\Omega)}
\ee
Similarly, interpolation between \eqref{6-1} and \eqref{6-3} yields that for every $\xi \in Z_3$, one has
\be \nonumber
\xi^2 \int_0^1 \left|\frac{d}{dr} ( r \what{\psi} )  \right|^2 \frac1r \, dr \leq C \Phi^{-\frac67} \int_0^1 |\what{\BF^*}|^2 r \, dr .
\ee
This means
\be \label{8-13}
\|\partial_z v^z_{med, 3} \|_{L^2(\Omega)} \leq C \Phi^{-\frac37} \|\BF^*_{med, 3} \|_{L^2(\Omega)}.
\ee
Hence the proof of Proposition \ref{medregularity2} is completed.
\end{proof}

Combining the conclusions in Propositions \ref{smallflux}, \ref{Bpropcase1-1}, \ref{Bpropcase2-1}, \ref{medregularity1},  and \ref{medregularity2} completes the proof for Part (a) of Theorem \ref{thm1}.

%%%%%%%%%%%%%%%%%%%%%%%%%%%%%%swirl velocity%%%%%%%%%%%%%%%%%%%%%%%%%%%%%%%%%%%%%%%%%%%%%

\section{Analysis on the linearized problem for swirl velocity}\label{sec-swirl}

In this section,  the existence and  uniform estimates for $\Bv^\theta= v^\theta \Be_\theta$ are established.    Assume that $\Bv$ is continuous, the compatibility conditions for $v^\theta$   implies that $v^\theta (0, z) = 0$. Hence the linearized problem for $v^\theta$ can be written as
\be \label{swirlsystem}
\left\{   \ba
& \bar U(r)  \frac{\partial v^\theta}{\partial z }  - \left[ \frac{1}{r} \frac{\partial }{\partial r} \left( r \frac{\partial v^\theta}{\partial r}\right) + \frac{\partial^2 v^\theta}{\partial z^2} - \frac{v^\theta}{r^2} \right] =  F^\theta \ \ \ \mbox{in}\ \ D, \\
& v^\theta(0, z) = 0, \\
& \frac{\partial v^\theta}{\partial r}  (1, z) = (1- \alpha) v^\theta (1, z).
\ea
\right.
\ee

%%%%%%%%%%%%%%%%%%%%%%%%%%%%%%%%%%swirl %%%%%%%%%%%%%%%%%%%%%%%%%%%%%%%%%%%%%%%
\begin{pro}\label{swirl}
Assume that $\BF^\theta = F^\theta \Be_\theta \in L^2 (\Omega)$ and $\alpha >  0$. The linear problem \eqref{swirlsystem} admits a unique solution $v^\theta$ satisfying
\be \label{swirl-1}
\| \Bv^{\theta} \|_{H^2(\Omega)} \leq C \left( 1 + \frac{1}{\alpha}  \right) \|\BF^\theta \|_{L^2 (\Omega)},
\ \ \ \ \ \|\partial_z \Bv^\theta \|_{L^2(\Omega)} \leq C \left(1 + \frac{1}{\alpha} \right)^{\frac12} \Phi^{-\frac12} \|\BF^\theta\|_{L^2(\Omega)}.
\ee
where $\Bv^\theta = v^\theta \Be_\theta$ and  the  constant $C$ is independent of $\BF^\theta$, $\Phi$, and $\alpha$.
\end{pro}

\begin{proof}
{\it Step 1. Existence of solutions}.
Taking the Fourier transform with respect to $z$ for the equation in \eqref{swirlsystem} yields that for every fixed $\xi \in \mathbb{R}$,  $\what{v^\theta}$  satisfies
\be \label{swirl-2}
i \xi \bar{U}(r) \widehat{v^\theta} - ( \mathcal{L} - \xi^2 )  \widehat{v^\theta} = \widehat{F^\theta}.
\ee
Furthermore, the boundary conditions for $\widehat{v^\theta}$ become
\be \label{swirl-3}
 \widehat{v^\theta} (0) = 0\ \ \ \ \mbox{and}\ \ \ \ \frac{d}{dr} \what{v^\theta} (1) = ( 1- \alpha ) \what{v^\theta} (1).
\ee

To prove the existence of solutions to the problem \eqref{swirl-2}-\eqref{swirl-3}, we do not deal with it directly. Instead, we consider an auxiliary  linear problem,
\be \label{swirl-4} \left\{
\ba
& i \xi \bar{U}(r) V^\theta - \Delta_4 V^\theta + \xi^2 V^\theta = \widetilde{F^\theta}\ \ \ \ \mbox{in} \ B_1^4(0), \\
& \frac{\partial V^\theta}{\partial \Bn } = -\alpha V^\theta \ \ \ \ \ \ \mbox{on}\ \partial B_1^4(0),
\ea
\right.
\ee
where  $B_1^4(0)$ is the unit ball centered at the origin in $\mathbb{R}^4$,
\be \nonumber
\Delta_4 = \sum_{i=1}^4 \partial_{x_i}^2 ,\ \ \ \ \ \mbox{and}\ \ \ \widetilde{F^\theta} (x_1, x_2, x_3, x_4) = \frac{\what{F^\theta}(r)}{r}, \ \ \ r= \left( \sum_{i=1}^4 x_i^2  \right)^\frac12.
\ee

For every $\varphi$, $\phi \in H^1(B_1^4(0))$, define
\be \nonumber
\mathcal{N} ( \varphi, \,  \phi) = \int_{B_1^4(0)} i\xi \bar{U}(r) \varphi \overline{\phi}  +
 \nabla_4 \varphi \cdot \nabla_4 \overline{\phi}+ \xi^2 \varphi \overline{\phi} \, dx +
\alpha \int_{\partial B_1^4(0)} \varphi \overline{\phi} \, dS ,
\ee
where $\bar{\phi}$ is the complex conjugate of $\phi$.
According to Lemma \ref{sobolev}, there exists a constant $c$, such that for every $\varphi \in H^1(B_1^4(0))$,
\be \nonumber
|\mathcal{N} (\varphi, \varphi) | \geq c \|\varphi\|_{H_1(B_1^4(0))}^2.
\ee
This means that  $\mathcal{N}(\cdot, \cdot)$ is strongly coercive. Hence, by Lax-Milgram theorem, for every $\widetilde{F^\theta} \in L^2(B_1^4(0))$, there exists a unique solution $V^\theta \in H^1(B_1^4(0))$, such that
\be \label{swirl-5}
\mathcal{N} (V^\theta, \, \phi )= \int_{B_1^4(0)} \widetilde{F^\theta}  \overline{\phi} \, dx,\ \ \ \ \text{for any} \,\, \phi \in H^1(B_1^4(0) ).
\ee
This implies that $V^\theta$ is a weak solution to \eqref{swirl-4}, and
\be \nonumber
\| V^\theta \|_{H^1(B_1^4(0))} \leq C \| \widetilde{F^\theta}\|_{L^2(\Omega)}.
\ee
Applying the regularity theory for elliptic equations (\hspace{1sp}\cite{ADN}) yields
\be \label{swirl-6} \ba
\| V^\theta \|_{H^2(B_1^4(0))} & \leq C \Big( \Phi |\xi| \|V^\theta \|_{L^2(B_1^4(0))} +  \|\widetilde{F^\theta}\|_{L^2(B_1^4(0))} +  \|V^\theta\|_{H^1(B_1^4(0))}\Big) \\
& \leq C (1 + \Phi |\xi| ) \| \widetilde{F^\theta} \|_{L^2(B_1^4(0))}.
\ea
\ee
Hence $V^\theta$ is a strong solution to the problem \eqref{swirl-4}. Moreover, due to the uniqueness of solutions and the fact that $\widetilde{F^\theta}$ is axisymmetric, $V^\theta$ is also axisymmetric, i. e., $V^\theta = V^\theta(r)$. Let $\what{v^\theta} = rV^\theta$. It can be verified that $
\mathcal{L} \what{v^\theta} = r \Delta_4 V^\theta,$ and then $\what{v^\theta}$ satisfies the equation \eqref{swirl-2}. Moreover,
\be \nonumber
 \what{v^\theta}(0) =0, \ \ \ \ \mbox{and}\ \  \ \frac{d}{dr}\what{v^\theta}(1)  = \frac{\partial V^\theta}{\partial \Bn} (1) + V^\theta (1) = (1- \alpha) \what{v^\theta}(1).
\ee
Consequently, $\what{v^\theta}$ is a solution to the problem \eqref{swirl-2}-\eqref{swirl-3}.

{\it Step 2.  Uniform estimates (independent of $\Phi$) for $\what{v^\theta}$}. It follows from \eqref{swirl-5} with $\phi = V^\theta$ that
\be \label{swirl-7-1} \ba
&  \int_{B_1^4(0)} i \xi \bar{U}(r) |V^\theta|^2  +  |\nabla_4 V^\theta|^2+ \xi^2  |V^\theta|^2 \, dx  +  \alpha \int_{\partial B_1^4(0)} |V^\theta|^2 \, dS
=   \int_{B_1^4(0)} \widetilde{F^\theta} \cdot \overline{V^\theta} \, dx.
\ea \ee
Since $V^\theta$ is axisymmetric, the equality \eqref{swirl-7-1} implies that
\be  \label{swirl-8}
\alpha |\what{v^\theta} (1)|^2 = \alpha  |V^\theta(1)|^2 \leq C \int_{B_1^4(0)} \left| \widetilde{F^\theta} \right| |V^\theta | \, dx
= C \int_0^1 |\what{F^\theta}| |\what{v^\theta}| r \, dr .
\ee

On the other hand, multiplying \eqref{swirl-2} by $r\overline{\widehat{v^\theta}}$ and integrating over
$[0, 1]$ yield
\be \label{swirl-9}
\int_0^1 \left| \frac{d}{dr} ( r \widehat{v^\theta} )  \right|^2 \frac1r \, dr + (\alpha-2)  |\what{v^\theta}(1)|^2  + \xi^2 \int_0^1 |\widehat{v^\theta}|^2 r\,dr
= \Re \int_0^1 \widehat{F^\theta} \overline{\widehat{v^\theta}} r \, dr
\ee
and
\be \label{swirl-10}
 \xi \int_0^1 \bar{U} (r) |\widehat{v^\theta} |^2 r \, dr
= \Im \int_0^1 \widehat{F^\theta} \overline{\widehat{v^\theta}} r\, dr.
\ee
According to  the equality \eqref{swirl-9} and the estimate \eqref{swirl-8}, it holds that
\be \label{swirl-11}
 \int_0^1 \left| \frac{d}{dr}( r \widehat{v^\theta} )\right|^2 \frac1r \, dr +  \alpha|\what{v^\theta} (1) |^2  + \xi^2 \int_0^1 |\what{v^\theta}|^2 r \, dr  \leq C \frac{1+ \alpha }{\alpha } \int_0^1 |\widehat{F^\theta} | |\widehat{v^\theta} | r \, dr .
\ee
This,  together with Lemma \ref{lemma1}, implies that
\be \label{swirl-12}
\int_0^1 \left| \frac{d \what{v^\theta} }{dr} \right|^2 r \, dr + \int_0^1 \frac{|\what{v^\theta}|^2}{r}  \, dr + (1 +  \alpha) |\what{v^\theta}(1)|^2 + \xi^2 \int_0^1 |\what{v^\theta}|^2 r \, dr  \leq C \frac{\alpha^2 + 1 }{\alpha^2}\int_0^1 |\widehat{F^\theta}|^2 r \, dr .
\ee
Hence one has
\be \label{swirl-12-1}
\|v^\theta \|_{H^1(\Omega)} \leq C \frac{ \alpha + 1 }{ \alpha }\|F^\theta \|_{L^2(\Omega)}.
\ee

Multiplying \eqref{swirl-11} by $\xi^2$ and applying \eqref{swirl-8} give
\be \label{swirl-12-2}
 \xi^2 \int_0^1 \left| \frac{d}{dr}( r \widehat{v^\theta} )\right|^2 \frac1r \, dr +  \alpha \xi^2 |\what{v^\theta} (1) |^2  + \xi^4 \int_0^1 |\what{v^\theta}|^2 r \, dr  \leq C \frac{1+ \alpha^2 }{\alpha^2 } \int_0^1 |\widehat{F^\theta} |^2  r \, dr .
\ee
This yields
\be \label{swirl-12-3}
\|\partial_z v^\theta\|_{H^1(\Omega)} \leq C \frac{ \alpha  + 1 }{ \alpha  } \|F^\theta\|_{L^2(\Omega)}.
\ee

Due to the fact that $\bar{U}(r) \geq \frac{\Phi}{\pi} (1 - r^2)$,  the equality \eqref{swirl-10} implies
\be \label{swirl-10-1}
\Phi |\xi| \int_0^1 ( 1 - r^2) |\what{v^\theta}|^2 r \, dr \leq C \left| \int_0^1 \widehat{F^\theta} \overline{\widehat{v^\theta}} r\, dr \right|.
\ee
It follows from Lemmas \ref{lemma1} and \ref{weightinequality},  and \eqref{swirl-11}-\eqref{swirl-10-1} that one has
\begin{equation} \label{swirl-15-1}
\begin{aligned}
\int_0^1 |\what{v^\theta}|^2 rdr \leq & C \int_0^1 (1-r^2) |\what{v^\theta}|^2 rdr \\
&\quad +  C \left( \int_0^1 (1-r^2) |\what{v^\theta}|^2 rdr \right)^{\frac23} \left(\int_0^1 \left| \frac{d}{dr} ( r \widehat{v^\theta} )  \right|^2 \frac1r \, dr\right)^{\frac13}\\
\leq & C \left( \int_0^1 (1-r^2) |\what{v^\theta}|^2 rdr \right)^{\frac23} \left(\int_0^1 \left| \frac{d}{dr} ( r \widehat{v^\theta} )  \right|^2 \frac1r \, dr\right)^{\frac13}\\
\leq & C \left( \frac{ \alpha  + 1}{ \alpha } \right)^{\frac13} (\Phi |\xi|)^{-\frac23} \int_0^1 |\widehat{F^\theta}| |\widehat{v^\theta}| r \, dr.
\end{aligned}
\end{equation}
Therefore, one has
\begin{equation} \label{swirl-15-2}
\begin{aligned}
\int_0^1 |\what{v^\theta}|^2 rdr \leq  C \left( \frac{\alpha  + 1 }{\alpha } \right)^{\frac23} (\Phi |\xi|)^{-\frac43} \int_0^1 |\widehat{F^\theta}|^2 r \, dr.
\end{aligned}
\end{equation}
This, together with \eqref{swirl-11}, yields
\be \label{swirl-15-3}
\ba
\int_0^1 \left| \frac{d}{dr} ( r \widehat{v^\theta} )  \right|^2 \frac1r \, dr + \xi^2 \int_0^1 |\widehat{v^\theta}|^2 r\,dr
\leq  & C \frac{ \alpha  +1 }{ \alpha  } \left(\int_0^1 |\widehat{F^\theta}|^2 r  dr \right)^{\frac12}\left(\int_0^1| \widehat{v^\theta} |^2 r \, dr\right)^{\frac12}\\
\leq &  C \left( \frac{  \alpha  + 1 }{ \alpha  } \right)^{\frac43} (\Phi |\xi|)^{-\frac23} \int_0^1 |\widehat{F^\theta} |^2 r \, dr .
\ea
\ee
Multiplying the equation \eqref{swirl-2} by $r ( \mathcal{L} - \xi^2) \overline{\widehat{v^\theta}}$ and integrating over $[0, 1]$ give
\be \label{swirl-7}
- \int_0^1 | ( \mathcal{L} - \xi^2) \widehat{v^\theta} |^2 r \, dr
= \Re \int_0^1 \widehat{F^\theta} ( \mathcal{L} - \xi^2) \overline{\widehat{v^\theta}} r \, dr
-  \frac{4 \Phi}{\pi}\frac{\alpha }{4 + \alpha } \xi \Im \int_0^1 r \widehat{v^\theta} \frac{d}{dr} ( r \overline{\widehat{v^\theta}}) \, dr .
\ee
Combining  \eqref{swirl-15-2} and \eqref{swirl-15-3} shows
\be \nonumber \ba
 \left| \frac{4 \Phi}{\pi} \frac{\alpha}{4+ \alpha  } \xi \Im \int_0^1  r \widehat{v^\theta} \frac{d}{dr} ( r \overline{\widehat{v^\theta}} ) \, dr  \right|
\leq & C \frac{\alpha}{4+\alpha} \Phi |\xi|  \left( \int_0^1 |\widehat{v^\theta}|^2 r  \, dr \right)^{\frac12}
\left( \int_0^1 \left| \frac{d}{dr} ( r \widehat{v^\theta})  \right|^2 \frac1r \, dr   \right)^{\frac12}\\
\leq & C   \int_0^1 |\widehat{F^\theta} |^2 r \, dr.
\ea \ee
Hence, by Young's inequality, one has
\be \label{swirl-16-1}
\int_0^1 | ( \mathcal{L} - \xi^2) \widehat{v^\theta} |^2 r \, dr
\leq  C   \int_0^1 |\widehat{F^\theta} |^2 r \, dr .
\ee
This means
\be \label{swirl-16-2}
\|(\mathcal{L} + \partial_z^2 ) v^\theta\|_{L^2(\Omega)} \leq C \|F^\theta \|_{L^2(\Omega)}.
\ee

Note that $\Bv^\theta$ satisfies
\be \label{swirl-17}
\Delta \Bv^\theta = ( \mathcal{L} + \partial_z^2) v^\theta \Be_\theta \ \ \ \mbox{in} \ \Omega.
\ee
It follows from the trace theorem for axisymmetric functions and the estimates \eqref{swirl-12-1}  that
\be \nonumber
\|\Bv^\theta\|_{H^{\frac32}(\partial \Omega)} \leq C  \|\Bv^\theta\|_{H^1(\Omega)} + C \|\partial_z \Bv^\theta \|_{H^1(\Omega)}  \leq C \frac{ \alpha +1}{ \alpha } \|\BF^\theta \|_{L^2(\Omega)}.
\ee
Applying the regularity theory for elliptic equations (\hspace{1sp}\cite{ADN}) yields
\be \label{swirl-18}
\|\Bv^\theta\|_{H^2 (\Omega)}
\leq C \| ( \mathcal{L} + \partial_z^2) v^\theta\|_{L^2(\Omega)} + C \|\Bv^\theta\|_{H^{\frac32} (\partial \Omega)}
 \leq C \frac{ \alpha + 1 }{ \alpha  } \|\BF^\theta\|_{L^2(\Omega)}.
\ee

Furthermore, interpolation between \eqref{swirl-15-2} and \eqref{swirl-15-3} gives
\be \label{swirl-19}
\xi^2 \int_0^1 |\what{v^\theta} |^2 r \, dr
\leq C \frac{ \alpha  + 1 }{ \alpha  } \Phi^{-1} \int_0^1 |\what{F^\theta} |^2 r \, dr .
\ee
This implies
\be \label{swirl-20}
\|\partial_z \Bv^\theta\|_{L^2(\Omega)} \leq C \left( \frac{\alpha  +  1 }{\alpha }\right)^{\frac12} \Phi^{-\frac12} \|\BF^\theta\|_{L^2(\Omega)}.
\ee
Therefore, the proof of the proposition is completed.
\end{proof}

Combining the conclusions in Sections 3 and  4 completes the proof for Part (b) of Theorem \ref{thm1}.

{
\begin{remark}\label{remark-nonzeroalpha}
When $\alpha = 0$,  it follows from \eqref{swirl-4} that some compatibility condition on $F^\theta$ is required for the existence of $v^\theta$ (also see \cite{WX-Navier}). On the other hand,  when  $F^\theta = 0$ and $\alpha  =0$, the homogeneous problem \eqref{swirlsystem} has the nonzero solution $v^\theta = \beta r$. Therefore, it is reasonable that one fails to get uniform estimates with respect to $\alpha$, when $\alpha$ is close to $0$. Furthermore, in order to get a uniform estimate for $\Bv^\theta$ in $H^2$, one needs to prescribe much stronger conditions on $\BF^\theta$ comparing with the case for flows periodic in the axial direction.
\end{remark}
}

%%%%%%%%%%%%%%%%%%%%%%%%%%%%Nonlinear without swirl%%%%%%%%%%%%%%%%%%%%%%%%%%%%%%55
\section{Nonlinear Structural Stability}\label{secnonlinear}
In this section, we  prove the existence of solutions for the nonlinear problem \eqref{NS}-\eqref{flux}, which gives the uniform nonlinear structural stability of Poiseuille flows.
Let $\Bv=\Bu-\bBU$ denote the perturbed  velocity. Then $\Bv$ satisfies the nonlinear perturbation system
 \be  \label{perturb}
\left\{ \ba
&\bBU \cdot \nabla \Bv + \Bv \cdot \nabla \bBU + (\Bv \cdot \nabla )\Bv - \Delta \Bv + \nabla P = \BF \ \ \ \mbox{in}\ \Omega, \\
& {\rm div}~\Bv = 0\ \ \ \mbox{in}\ \Omega, \\
\ea
\right.
\ee
supplemented with the following boundary conditions
\begin{equation} \label{perturb2}
\Bv \cdot \Bn  = 0,\ \ \ \ \ \  2 \Bn \cdot D(\Bv) \cdot \Bt +  \alpha \Bv \cdot \Bt= 0 \ \ \ \mbox{on}\ \partial \Omega,
\end{equation}
and flux constraint
\begin{equation}\label{perturb3}
  \int_{\Sigma} \Bv \cdot \Bn \, dS = 0.
\end{equation}

\subsection{Existence and uniqueness for the problem without swirl velocity}
In this subsection, we prove the existence and uniqueness of strong axisymmetric solution when $\BF$ has no swirl. Before the proof, let us collect some estimates for solutions to the linearized problem.

\begin{lemma}\label{linearized-estimates-new}
Assume that $\BF= \BF(r, z) \in L^2(\Omega)$ and $\Phi \geq \Phi_0$ is sufficiently large.

(a) The unique solution $\Bv$ to the  linear problem \eqref{linearizedNS}-\eqref{fluxBC} satisfies
\be \label{9-1}
\|\Bv^*\|_{H^{\frac54}(\Omega)} \leq C_1 \Phi^{\frac{1}{16}}\|\BF^*\|_{L^2(\Omega)},
\ee
and
\be \label{9-2}
\|\Bv^r\|_{H^{\frac54}(\Omega)} \leq C_1 \Phi^{-\frac{23}{160}}\|\BF^*\|_{L^2(\Omega)},\ \ \ \ \ \ \|\partial_z \Bv^z \|_{H^{\frac14}(\Omega)} \leq C_1 \Phi^{-\frac{29}{112}}\|\BF^* \|_{L^2(\Omega)}.
\ee

(b) If $F^z = 0$, the solution $\Bv$ to \eqref{linearizedNS}-\eqref{fluxBC} satisfies
\be \label{9-3}
\|\Bv^* \|_{H^{\frac54}(\Omega)} \leq C_1 \Phi^{-\frac{3}{32}} \|F^r\|_{L^2(\Omega)}.
\ee

(c) If $\alpha \geq \alpha_0 > 0$, the solution $\Bv$ to \eqref{linearizedNS}-\eqref{fluxBC} satisfies
\be \label{9-6}
\|\Bv^\theta\|_{H^{\frac54}(\Omega)} \leq C_2\left( 1 + \frac{1}{\alpha_0} \right) \|F^\theta\|_{L^2(\Omega)}, \ \ \
\|\partial_z \Bv^\theta \|_{H^{\frac14}(\Omega)} \leq C_2 \left( 1 + \frac{1}{\alpha_0} \right)^{\frac58} \Phi^{-\frac38} \|F^\theta \|_{L^2(\Omega)}.
\ee
Here $C_1$, $C_2$ are constants independent of $\BF$, $ \Phi$, and $\alpha$.
\end{lemma}

\begin{proof}
(a). The interpolation between \eqref{estuniformlinear} and \eqref{estimatelinear} yields \eqref{9-1}, and the interpolation between \eqref{estimatelinear} and \eqref{estimatelinear-Philarge} yields \eqref{9-2}.

(b). For every $\xi \in \mathbb{R}$, as proved in \eqref{3-7}, one has
\be \label{9-9}
\xi \int_0^1 \frac{\bar{U}(r)}{r} \left|\frac{d}{dr} ( r \hat{\psi})   \right|^2 \, dr   + \xi^3 \int_0^1 \bar{U}(r) |\hat{\psi}|^2 r \, dr
= -\Re \int_0^1 \xi \what{F^r} \overline{\hat{\psi}} r\, dr.
\ee
This, together with Lemma \ref{lemmaHLP} and \eqref{3-7-1}, gives
\be \nonumber  \ba
\Phi \int_0^1 \left| \frac{d}{dr} (r \hat{\psi})  \right|^2 \frac{1 - r^2 }{r } \, dr
& \leq C \int_0^1 \left| \what{F^r}  \right| |\hat{\psi}| r \, dr  \\
& \leq C \left( \int_0^1 |\what{F^r} |^2 r \, dr \right)^{\frac12} \left( \int_0^1 \left|\frac{d}{dr}(r \hat{\psi})  \right|^2 \frac{1- r^2 }{r} \, dr \right)^{\frac12} .
\ea \ee
Hence one has
\be \label{9-11}
\int_0^1 \left| \frac{d}{dr} (r \hat{\psi})  \right|^2 \frac{1 - r^2 }{r } \, dr
\leq C \Phi^{-2} \int_0^1 |\what{F^r}|^2r \, dr.
\ee

On the other hand, it follows from the estimates in Section 3 that
\be \label{9-11-1}
\int_0^1 |\mathcal{L} \hat{\psi}|^2 r \, dr + \xi^2 \int_0^1  \left| \frac{d}{dr} (r \hat{\psi})  \right|^2 \frac1r \, dr \leq C \int_0^1 |\what{F^r}|^2r \, dr.
\ee
Then by Lemma \ref{lemma1} and \ref{weightinequality},
\be \nonumber \ba
\int_0^1 \left| \frac{d}{dr}(r \hat{\psi})  \right|^2 \frac1r \, dr
& \leq C \left( \int_0^1 \left|  \frac{d}{dr}(r \hat{\psi}) \right|^2 \frac{1-r^2}{r} \, dr  \right)^{\frac23} \left( \int_0^1 |\mathcal{L} \hat{\psi}|^2 r \, dr  \right)^{\frac13} \\
& \leq C \Phi^{-\frac43} \int_0^1 |\what{F^r}|^2 r \, dr .
\ea \ee
This implies
\be \label{9-15}
\|v^z \|_{L^2(\Omega)}^2  = \int_{-\infty}^{+\infty} \int_0^1 \left| \frac{d}{dr}(r \hat{\psi} ) \right|^2 \frac1r \, dr d\xi \leq C \Phi^{-\frac43} \| F^r \|_{L^2(\Omega)}^2.
\ee

Similarly, the equality \eqref{9-9}, together with \eqref{9-11}, gives
\be \label{9-16}
\Phi \xi^2 \int_0^1 (1  - r^2) |\hat{\psi} |^2 r \, dr
\leq C \Phi^{-1} \int_0^1 |\what{F^r}|^2 r \, dr.
\ee
With the aid of Lemmas \ref{lemma1} and \ref{weightinequality}, it follows from \eqref{9-11-1} again that one has
\be \nonumber \ba
\xi^2 \int_0^1 |\hat{\psi}|^2 r \, dr
& \leq C \left( \xi^2 \int_0^1 (1  - r^2) |\hat{\psi}|^2 r \, dr  \right)^{\frac23} \left(  \xi^2 \int_0^1 \left|\frac{d}{dr}(r \hat{\psi})  \right|^2 \frac1r \, dr \right)^{\frac13} \\
& \leq C \Phi^{-\frac43} \int_0^1 |\what{F^r}|^2 r \, dr .
\ea
\ee
This implies
\be \label{9-18}
\| v^r \|_{L^2(\Omega)}^2 = \int_{-\infty}^{+ \infty} \int_0^1 \xi^2 |\hat{\psi}|^2 r \, dr d\xi \leq C \Phi^{-\frac43} \| F^r\|_{L^2(\Omega)}^2.
\ee
The estimate \eqref{9-15} together with \eqref{9-18} gives
\be \label{9-19}
\|\Bv^* \|_{L^2(\Omega)} \leq C \Phi^{-\frac23} \| F^r\|_{L^2(\Omega)}.
\ee
The interpolation between \eqref{9-19} and \eqref{estimatelinear} yields \eqref{9-3}.

The estimates in Part (c) are implied by Proposition \ref{swirl}.
\end{proof}

{Lemma \ref{linearized-estimates-new} shows that there are  better estimate for  $v^r$, $\partial_z v^z$, and $\partial_z v^\theta$ as $\Phi$ is large. On the other hand, each nonlinear term in \eqref{nonlinearperturb} contains either $v^r$ or $z$-derivatives of $\Bv$. These two key facts enable us to get the solutions to the nonlinear problem when $\Phi$ is large.  This is one of major new observations of this paper comparing with the study for flows periodic in the axial direction in \cite{WX-Navier}. }

For any given $\BF = \BF(r, z) \in L^2(\Omega)$, the linear problem \eqref{linearizedNS}-\eqref{fluxBC} admits a unique axisymmetric solution $\Bv$. We denote $\Bv= \mathcal{T} \BF$.

\begin{proof}[\bf{Proof of Theorem \ref{mainthm} Part (b).}] We use the iteration method to prove the existence. The proof is divided into $3$ steps.

{\it Step 1. Iteration Scheme.} Let $\Bv_0 = \mathcal{T} \BF$. Since $F^\theta = 0$, one has $v_0^\theta = 0$. For every $j \geq 0$,
let
$\Bv_{j+1}$ be the solution to the iteration problem
\be \label{10-1}
\Bv_{j+1} = \mathcal{T} \BF + \mathcal{T} ( F_j^r \Be_r + F_j^z \Be_z + F_j^\theta \Be_\theta ),
\ee
where
\be \nonumber
F_j^r = -\left( v^r_j \partial_r v^r_j + v^z_j \partial_z v^r_j    \right) + \frac{(v_j^\theta)^2 }{r},
\quad
F_j^z = - \left( v^r_j  \partial_r v^z_j + v^z_j \partial_z v^z_j    \right),
\ee
and
\be \nonumber
F_j^\theta = -(v^r_j \partial_r v^\theta_j  + v^z_j \partial_z v^\theta_j) - \frac{v^r_j v^\theta_j}{r}.
\ee

{\it Step 2. Mathematical induction and the existence of solution.}
Denote $\Bv^r = v^r \Be_r$, $\Bv^z = v^z \Be_z$ and $\Bv^\theta = v^\theta \Be_\theta$. Denote
\be \nonumber
\mathcal{J} = \left\{ \Bv = \Bv(r, z) \, \left|  \ba & v^\theta = 0,\ \ \ \ \ \|\Bv \|_{H^{\frac54}(\Omega)} \leq 2C_1 \Phi^{\frac{1}{16}} \|\BF\|_{L^2(\Omega)}, \\
& \|\Bv^r\|_{H^{\frac54}(\Omega)} + \|\partial_z \Bv^z\|_{H^{\frac14}( \Omega) }  \leq \Phi^{-\frac{1}{10} } \|\BF\|_{L^2(\Omega)}    \ea   \right. \right\},
\ee
where $C_1$ is the constant indicated in Lemma \ref{linearized-estimates-new}.

Assume that $\|\BF\|_{L^2(\Omega)} \leq \Phi^{\frac{1}{40}}$. According to Lemma \ref{linearized-estimates-new}, one has $\Bv_0 \in \mathcal{J}$. Assume that $\Bv_j \in \mathcal{J}$, we prove that $\Bv_{j+1} \in \mathcal{J}$. Since $F_j^\theta = 0$, it holds that
$v_{j+1}^\theta  = 0$. On the other hand,
by H\"older inequality and Sobolev embedding inequality, one has
\be \nonumber \ba
\|F_j^r \|_{L^2(\Omega)}  & \leq  \|v_j^r \partial_r v_j^r \|_{L^2(\Omega)} + \|v_j^z \partial_z v_j^r\|_{L^2(\Omega)} \\
& \leq C \|v_j^r \|_{L^{12}(\Omega)} \|\partial_r v_j^r \|_{L^{\frac{12}{5}} (\Omega)} + C \|v_j^z \|_{L^{12}(\Omega)} \|\partial_z v_j^r\|_{L^{\frac{12}{5}} (\Omega) } \\
& \leq C \|\Bv_j^r \|_{H^{\frac54}(\Omega)}^2 + C \|\Bv_j^z \|_{H^{\frac54}(\Omega)} \|\Bv_j^r \|_{H^{\frac54} (\Omega) } \\
& \leq C \Phi^{-\frac15} \|\BF\|_{L^2(\Omega)}^2 + C \Phi^{-\frac{3}{80}} \|\BF\|_{L^2(\Omega)}^2 \\
& \leq C \Phi^{-\frac{1}{80}} \|\BF\|_{L^2(\Omega)}
\ea
\ee
and
\be \nonumber
\ba
\|F_j^z \|_{L^2(\Omega)} & \leq \|v_j^r \partial_r v_j^z \|_{L^2(\Omega)} + \|v_j^z \partial_z v_j^z \|_{L^2(\Omega)} \\
& \leq C \|\Bv_j^r \|_{H^{\frac54}(\Omega)} \|\Bv_j^z \|_{H^{\frac54}(\Omega)} + C \|\Bv_j^z \|_{H^{\frac54}(\Omega)} \|\partial_z \Bv_j^z \|_{H^{\frac14}(\Omega)} \\
& \leq C \Phi^{-\frac{1}{10}}\Phi^{\frac{1}{16}} \|\BF\|_{L^2(\Omega)}^2 + C \Phi^{\frac{1}{16}} \Phi^{-\frac{1}{10}} \|\BF\|_{L^2(\Omega)}^2 \\
& \leq C \Phi^{-\frac{1}{80}} \|\BF\|_{L^2(\Omega)}.
\ea
\ee

By virtue of Lemma \ref{linearized-estimates-new}, it holds that
\be \label{10-5}
\|\Bv_{j+1} \|_{H^{\frac54}(\Omega)} \leq C_1 \left(  \Phi^{\frac{1}{16}} \|\BF\|_{L^2(\Omega)} +  C \Phi^{-\frac{1}{80}}\|\BF\|_{L^2(\Omega)} \right),
\ee
\be \label{10-6}
\|\Bv_{j+1}^r \|_{H^{\frac54}(\Omega)} \leq C_1 \Phi^{-\frac{23}{160}} \left( \|\BF\|_{L^2(\Omega)} + C \Phi^{-\frac{1}{80}} \|\BF\|_{L^2(\Omega)}      \right),
\ee
and
\be  \label{10-7}
\| \partial_z \Bv_{j+1}^z  \|_{H^{\frac14}(\Omega)} \leq C_1 \Phi^{-\frac{29}{112}} \left( \|\BF\|_{L^2(\Omega)} + C \Phi^{-\frac{1}{80}} \|\BF\|_{L^2(\Omega)}      \right).
\ee
Hence there exists a constant $\Phi_1$ such that if $\Phi \geq \Phi_1$, the solution $\Bv_{j+1} \in \mathcal{J}$. By mathematical induction, $\Bv_j \in \mathcal{J}$ for every $j \in \mathbb{N}$. Due to the uniform bound for $\Bv_j$, there exists a  subsequence which is still denoted by  $\{ \Bv_j \}$ and converges to $ \Bv\in \mathcal{J}$ weakly in $H^{\frac54}(\Omega)$. Furthermore, it holds that
\be \nonumber
\|\Bv\|_{H^{\frac54} (\Omega)} \leq 2C_1 \Phi^{\frac{1}{16}} \|\BF\|_{L^2(\Omega)} \quad \text{and}\ \ \ \
 \|\Bv^r \|_{H^{\frac54}(\Omega)} + \|\partial_z \Bv^z \|_{H^{\frac14}(\Omega)} \leq \Phi^{-\frac{1}{10}} \|\BF\|_{L^2(\Omega)}.
\ee
In fact, by virtue of Theorem \ref{thm1}, one has
\be \nonumber
\|\Bv_{j+1}\|_{H^2(\Omega)} \leq C \Phi^{\frac14} \|\BF\|_{L^2(\Omega)} + C \Phi^{\frac14} \|  F_j^r \Be_r + F_j^z \Be_z \|_{L^2(\Omega)}
\leq C \Phi^{\frac14} \|\BF\|_{L^2(\Omega)}
\ee
and
\be \label{10-16}
\|\Bv_{j+1} \|_{H^1(\Omega)} \leq C \|\BF\|_{L^2(\Omega)} + C \|F_j^r\Be_r + F_j^z \Be_z \|_{L^2(\Omega)} \leq C \|\BF\|_{L^2(\Omega)}.
\ee
Hence it holds that
\be \nonumber
\|\Bv\|_{H^2(\Omega)} \leq C \Phi^{\frac14} \|\BF\|_{L^2(\Omega)},\ \ \ \ \ \|\Bv\|_{H^1(\Omega)} \leq C \|\BF\|_{L^2(\Omega)}.
\ee

On the other hand, since $\Bv_{j+1}$ is a solution to \eqref{10-1}, one can verify that
\be \label{10-11}
\curl~\left[ (\bBU \cdot \nabla )\Bv_{j+1} + ( \Bv_{j+1} \cdot \nabla ) \bBU - \Delta \Bv_{j+1} + ( \Bv_j \cdot \nabla ) \Bv_j \right] = \curl~\BF\ \ \ \mbox{in}\ \Omega.
\ee
Taking the limit of the equation \eqref{10-11} yields
\be \nonumber
\curl~ \left[(\bBU \cdot \nabla ) \Bv + (\Bv \cdot \nabla )\bBU - \Delta \Bv + (\Bv\cdot \nabla ) \Bv \right] = \curl~ \BF\ \ \ \mbox{in}\ \Omega.
\ee
Due to the fact that $\Bv \in H^2(\Omega)$ and $\BF \in L^2(\Omega)$, there exists a function $P$ with $\nabla P \in L^2(\Omega)$, such that
\be \label{10-12}
(\bBU \cdot \nabla ) \Bv + (\Bv \cdot \nabla )\bBU - \Delta \Bv + (\Bv\cdot \nabla ) \Bv + \nabla P = \BF\ \ \ \mbox{in}\ \Omega.
\ee

{\it Step 3. Uniqueness. } Suppose that $\Bv$, $\tilde{\Bv} \in \mathcal{J}$ are two solutions of the nonlinear problem \eqref{perturb}-\eqref{perturb3}. Then
\be \nonumber
\Bv - \tilde{\Bv} = \mathcal{T} \tilde{\BF},
\ee
where
\be \nonumber \ba
\tilde{\BF} = \tilde{F}^r \Be_r + \tilde{F}^z \Be_z =& - [(v^r \partial_r v^r + v^z \partial_z v^r ) - (\tilde{v}^r \partial_r \tilde{v}^r + \tilde{v}^z \partial_z \tilde{v}^r) ]\Be_r  \\
&\ \ \ \ \ \ -  [(v^r \partial_r v^z + v^z \partial_z v^z ) - (\tilde{v}^r \partial_r \tilde{v}^z + \tilde{v}^z \partial_z \tilde{v}^z )]\Be_z .
\ea
\ee
Note that
\be \nonumber \ba
\|\tilde{F}^r \|_{L^2(\Omega)} & \leq \|( v^r - \tilde{v}^r) \partial_r v^r  \|_{L^2(\Omega)} + \|\tilde{v}^r (\partial_r v^r - \partial_r \tilde{v}^r ) \|_{L^2(\Omega)} \\
& \ \ \ \ \ \ \ + \| (v^z - \tilde{v}^z ) \partial_z v^r \|_{L^2(\Omega)} + \|\tilde{v}^z (\partial_z v^r - \partial_z \tilde{v}^r) \|_{L^2(\Omega)} \\
& \leq C \|\Bv^r - \tilde{\Bv}^r \|_{H^{\frac54}(\Omega)} \left( \|\Bv^r \|_{H^{\frac54}(\Omega)} + \|\tilde{\Bv}^r \|_{H^{\frac54}(\Omega)} + \|\tilde{\Bv}^z \|_{H^{\frac54}(\Omega)} \right) \\
&\ \ \ \ \ \ + C \|\Bv^z - \tilde{\Bv}^z \|_{H^{\frac54}(\Omega)} \|\Bv^r \|_{H^{\frac54}(\Omega)} \\
& \leq C \Phi^{\frac{7}{80} } \|\Bv^r - \tilde{\Bv}^r \|_{H^{\frac54} (\Omega)} + C \Phi^{-\frac{3}{40}}\|\Bv^z - \tilde{\Bv}^z \|_{H^{\frac54}(\Omega)}
\ea \ee
and
\be \nonumber
\ba
\|\tilde{F}^z \|_{L^2(\Omega)} & \leq \| (v^r - \tilde{v}^r ) \partial_r v^z \|_{L^2(\Omega)} + \|\tilde{v}^r (\partial_r v^z - \partial_r \tilde{v}^z ) \|_{L^2(\Omega)} \\
& \ \ \ \ \ \ + \|(v^z - \tilde{v}^z ) \partial_z v^z \|_{L^2(\Omega)} + \|\tilde{v}^z (\partial_z v^z - \partial_z \tilde{v}^z \|_{L^2(\Omega)} \\
& \leq C \|\Bv^r - \tilde{\Bv}^r \|_{H^{\frac54} (\Omega)} \|\Bv^z \|_{H^{\frac54}(\Omega)}  + C \|\partial_z \Bv^z - \partial_z \tilde{\Bv}^z \|_{H^{\frac14}(\Omega)} \|\tilde{\Bv}^z  \|_{H^{\frac54} (\Omega)} \\
&\ \ \ \ \ + C \|\Bv^z - \tilde{\Bv}^z \|_{H^{\frac54}(\Omega)} \left( \|\tilde{\Bv}^r \|_{H^{\frac54}(\Omega)} + \| \partial_z v^z \|_{H^{\frac14} (\Omega)} \right) \\
& \leq C \Phi^{\frac{7}{80} } \left( \| \Bv^r - \tilde{\Bv}^r \|_{H^{\frac54}(\Omega)} +  \|\partial_z \Bv^z - \partial_z \tilde{\Bv}^z \|_{H^{\frac14}(\Omega)} \right) + C \Phi^{-\frac{3}{40}} \|\Bv^z - \tilde{\Bv}^z \|_{H^{\frac54}(\Omega)}.
\ea
\ee
By virtue of Lemma \ref{linearized-estimates-new},
\be \nonumber \ba
& \Phi^{\frac15} \|\Bv^r - \tilde{\Bv}^r \|_{H^{\frac54} (\Omega)} + \|\Bv^z - \tilde{\Bv}^z \|_{H^{\frac54} (\Omega)} +
\Phi^{\frac15} \|\partial_z \Bv^z - \partial_z \tilde{\Bv}^z \|_{H^{\frac14} (\Omega) } \\
\leq & C ( \Phi^{\frac15} \Phi^{-\frac{23}{160} } + \Phi^{\frac{1}{16} } + \Phi^{\frac15} \Phi^{-\frac{29}{112}} ) \left(  \Phi^{\frac{7}{80}- \frac{1}{5} } \Phi^{\frac15} \|\Bv^r - \tilde{\Bv}^r \|_{H^{\frac54}(\Omega)} \right.  \\
 &\ \ \ \ \ \ \ \ \ \ \ \left. + \Phi^{-\frac{3}{40} } \| \Bv^z - \tilde{\Bv}^z \|_{H^{\frac54}(\Omega)}
 +  \Phi^{ \frac{7}{80} -\frac15 } \Phi^{\frac15} \|\partial_z \Bv^z - \partial_z \tilde{\Bv}^z \|_{H^{\frac14} (\Omega)} \right)\\
 \leq & C \Phi^{- \frac{1}{20}} \left( \Phi^{\frac15} \|\Bv^r - \tilde{\Bv}^r \|_{H^{\frac54}(\Omega)} + \Phi^{\frac15} \|\partial_z \Bv^z - \partial_z \tilde{\Bv}^z \|_{H^{\frac14} (\Omega)} \right) + C \Phi^{-\frac{1}{80}}\| \Bv^z - \tilde{\Bv}^z \|_{H^{\frac54}(\Omega)}.
\ea \ee
 This implies  $\Bv^r= \tilde{\Bv}^r$ and  $\Bv^z = \tilde{\Bv}^z$ when $\Phi$ is large enough.  Hence the uniqueness is proved. This finishes the proof for part (b) of  Theorem \ref{mainthm}.

{\bf Proof  of Theorem \ref{mainthm} Part (a).}  Without loss of generality, it suffices to consider the case with $\Phi \leq \Phi_1$. Suppose that $\Bv$ is a solution to the nonlinear perturbation system \eqref{perturb}-\eqref{perturb3}, $\Bv$ is a solution to the following equation,
\be \label{10-21}
\Bv = \mathcal{T}\BF - \mathcal{T} [(\Bv \cdot \nabla )\Bv].
\ee
For every $\Bv, \Bw \in H^{\frac54}(\Omega)$, one has
\be \nonumber
\|(\Bv \cdot \nabla ) \Bw\|_{L^2(\Omega)} \leq C \|\Bv\|_{L^{12}(\Omega)} \|\nabla \Bw \|_{L^{\frac{12}{5}} (\Omega)} \leq C \|\Bv\|_{H^{\frac54}(\Omega)} \|\Bw\|_{H^{\frac54}(\Omega)} .
\ee
This, together with Theorem \ref{thm1}, yields
\be \nonumber
\|\mathcal{T} (\Bv\cdot \nabla )\Bw \|_{H^{\frac54}(\Omega)}  \leq C \|\Bv\|_{H^{\frac54}(\Omega)} \|\Bw\|_{H^{\frac54}(\Omega)}.
\ee
Hence, by Lemma \ref{nonlinear} (fixed point theorem), there exists a small constant $\epsilon_0$ such that if $\|\BF\|_{L^2(\Omega)} \leq \epsilon_0$, the equation
\eqref{10-21} has a unique solution $\Bv$ satisfying
\be \nonumber
\|\Bv\|_{H^{\frac54}(\Omega)} \leq C \|\BF\|_{L^2(\Omega)} \leq C \epsilon_0.
\ee
Moreover, according to Theorem \ref{thm1}, one has
\be \nonumber
\|\Bv\|_{H^2 (\Omega) } \leq C (1 + \Phi^{\frac14} ) ( \|\BF\|_{L^2(\Omega)} + \|( \Bv\cdot \nabla )\Bv\|_{L^2(\Omega)} )
\leq C (1 + \Phi^{\frac14} ) \|\BF\|_{L^2(\Omega)}.
\ee
This finished the proof of Theorem \ref{mainthm}.
\end{proof}

\subsection{Existence and Uniqueness of solution for the problem with swirl velocity}
In this subsection, we prove the existence and uniqueness of strong axisymmetric solutions when $\BF$ has nonzero swirl component. The proof is in the same spirit as that in Subsection 5.1, whereas we assume that $\alpha \geq \alpha_0 > 0$, due to the reason stated in Remark \ref{remark-nonzeroalpha}.
\begin{proof}[\bf{Proof of Theorem \ref{mainthm2} Part (b).}]
As in the proof of Theorem \ref{mainthm}, we also use the iteration method to prove the existence.

{\it Step 1. Iteration scheme.}\ Given $\BF = \BF(r, z) \in L^2(\Omega)$. As in the proof of Theorem \ref{mainthm}, let
\be \label{10-31}
\Bv_0 = \mathcal{T} \BF \ \ \ \ \mbox{and}\ \ \ \ \Bv_{j+1} = \mathcal{T}\BF + \mathcal{T} ( F_j^r \Be_r + F_j^z \Be_z + F_j^\theta \Be_\theta) .
\ee

{\it Step 2. Mathematical induction and the existence of solution.}
Set
\be \nonumber
\mathcal{K } = \left\{  \Bv(r, z) \left|  \ba & \|\Bv^* \|_{H^{\frac54}(\Omega)}
\leq 2C_1 \Phi^{\frac{1}{16} } \|\BF\|_{L^2(\Omega)} ,\ \ \ \ \|\Bv^\theta \|_{H^{\frac54}(\Omega)} \leq 2C_2 \left( 1 + \frac{1}{\alpha_0} \right) \|\BF \|_{L^2(\Omega)}.\\
& \|\Bv^r \|_{H^{\frac54}(\Omega)} + \|\partial_z \Bv^z \|_{H^{\frac14}(\Omega)} \leq \Phi^{-\frac{1}{10}} \|\BF\|_{L^2(\Omega)} ,
\ \|\partial_z \Bv^\theta \|_{H^{\frac14}(\Omega)} \leq \Phi^{-\frac14} \| \BF\|_{L^2(\Omega)}.
\ea             \right. \right\}.
\ee

Assume that $\|\BF\|_{L^2(\Omega)} \leq \Phi^{\frac{1}{40}}$. According to Lemma \ref{linearized-estimates-new}, one has $\Bv_0 \in \mathcal{K}$. Assume that $\Bv_j \in \mathcal{K}$, we prove that $\Bv_{j+1} \in \mathcal{K}$.
By H\"older inequality and Sobolev embedding inequality, one has
\be \nonumber
\|F_j^z \|_{L^2(\Omega)} \leq \|v_j^r \partial_r v_j^z  \|_{L^2(\Omega)} + \|v_j^z \partial_z v_j^z \|_{L^2(\Omega)} \leq C \Phi^{-\frac{1}{80}} \|\BF\|_{L^2(\Omega)},
\ee
\be \nonumber
\ba
\|F_j^r \|_{L^2(\Omega)} & \leq   \|v_j^r \|_{L^{12}(\Omega)} \|\partial_r v_j^r \|_{L^{\frac{12}{5}} (\Omega)} +
 \| v_j^z \|_{L^{12}(\Omega)} \| \partial_z v_j^r \|_{L^{\frac{12}{5}}(\Omega)} +  \|v_j^\theta\|_{L^{12}(\Omega)} \left\|  \frac{v_j^\theta}{r} \right\|_{L^{\frac{12}{5}} (\Omega)} \\
& \leq C \|\Bv_j^r\|_{H^{\frac54}(\Omega)}^2 + C \|\Bv_j^z \|_{H^{\frac54}(\Omega)}\|\Bv_j^r \|_{H^{\frac54}(\Omega)}
+ C \|\Bv_j^\theta\|_{H^{\frac54}(\Omega)}^2 \\
& \leq C \Phi^{-\frac15} \|\BF\|_{L^2(\Omega)}^2 + C \Phi^{-\frac{3}{80}} \|\BF\|_{L^2(\Omega)}^2 + C \| \BF\|_{L^2(\Omega)}^2 \\
& \leq C \Phi^{\frac{1}{40}} \|\BF\|_{L^2(\Omega)},
\ea
\ee
and
\be \nonumber \ba
\|F_j^\theta\|_{L^2(\Omega)}&  \leq \|v_j^r \|_{L^{12}(\Omega)} \|\partial_r v_j^\theta \|_{L^{\frac{12}{5}} (\Omega)}
+ \|v_j^z \|_{L^{12}(\Omega)} \| \partial_z v_j^\theta \|_{\frac{12}{5} (\Omega) } + \|v_j^r \|_{L^{12}(\Omega)} \left\| \frac{v_j^\theta}{r} \right\|_{L^{\frac{12}{5}} (\Omega)}\\
&\leq C \|\Bv_j^r \|_{H^{\frac54}(\Omega)} \|\Bv_j^\theta \|_{H^{\frac54} (\Omega) } +
C \|\Bv_j^z \|_{H^{\frac54} (\Omega) } \|\partial_z \Bv_j^\theta \|_{H^{\frac14} (\Omega)} \\
& \leq C \Phi^{-\frac{1}{10}} \|\BF\|_{L^2(\Omega)}^2 + C \Phi^{\frac{1}{16}} \Phi^{-\frac14} \|\BF\|_{L^2(\Omega)}^2 \\
& \leq C \Phi^{-\frac{3}{40}} \|\BF\|_{L^2(\Omega)}.
\ea
\ee
Hence by virtue of Lemma \ref{linearized-estimates-new}, it holds that
\be \nonumber  \ba
\|\Bv_{j+1}^* \|_{H^{\frac54}(\Omega)}
&  \leq  C_1 \Phi^{-\frac{3}{32}}
\left( \| F^r \|_{L^2(\Omega)} + C \Phi^{\frac{1}{40}} \|\BF\|_{L^2(\Omega)} \right) \\
 &\ \ \ \ \ \ \ + C_1 \Phi^{\frac{1}{16}} \left( \|F^z\|_{L^2(\Omega)} + C \Phi^{-\frac{1}{80}} \|\BF\|_{L^2(\Omega)} \right) \\
 & \leq C_1 \Phi^{\frac{1}{16}} \|\BF\|_{L^2(\Omega)} + C \Phi^{-\frac{11}{160}} \|\BF\|_{L^2(\Omega)}
\ea
\ee
and
\be \nonumber
\|\Bv_{j+1}^\theta \|_{H^{\frac54} (\Omega)}
\leq C_2 \left( 1  + \frac{1}{\alpha_0} \right) \left(  \|F^\theta \|_{L^2(\Omega)} + C \Phi^{-\frac{3}{40}} \| \BF\|_{L^2(\Omega)}                \right).
\ee
Furthermore, one has
\be \nonumber
\|\Bv_{j+1}^r \|_{H^{\frac54} (\Omega)}
\leq C_1 C \Phi^{-\frac{23}{160} }   \Phi^{\frac{1}{40}} \|\BF\|_{L^2(\Omega)},\quad
\|\partial_z \Bv_{j+1}^z \|_{H^{\frac14}(\Omega)} \leq C_1 C \Phi^{-\frac{29}{112}} \Phi^{\frac{1}{40}} \|\BF\|_{L^2(\Omega)},
\ee
and
\be \nonumber
\|\partial_z \Bv_{j+1}^\theta \|_{H^{\frac14}(\Omega) } \leq C_2 \left( 1 + \frac{1}{\alpha_0} \right)^{\frac58} C\Phi^{-\frac38} \| \BF\|_{L^2(\Omega)}.
\ee
Hence there exists a sufficiently large constant $\Phi_2$ such that if $\Phi \geq \Phi_2$, then the solution $\Bv_{j+1} \in \mathcal{K}$. By mathematical induction, $\Bv_j \in \mathcal{K}$ for every $j \in \mathbb{N}$. And there exists a subsequence which is still denoted by  $\Bv_j$ and converges to $\Bv \in \mathcal{K}$ weakly in $H^{\frac54}(\Omega)$. Following the
same lines as in the proof of Theorem 1.2, one can prove that $\Bv$ is the desired solution.

{\it Step 3. Uniqueness.}\ Suppose that $\Bv, \tilde{\Bv}\in \mathcal{K}$ are two solutions of the nonlinear perturbation problem \eqref{perturb}-\eqref{perturb3}. Then $\Bv - \tilde{\Bv} = \mathcal{T} \tilde{\BF}$ where
$
\tilde{\BF} = \tilde{F}^r \Be_r + \tilde{F}^z \Be_z + \tilde{F}^\theta \Be_\theta,
$
with
\be \nonumber
\tilde{F}^r = -\left( v^r \partial_r v^r + v^z \partial_z v^z - \frac{(v^\theta)^2 }{r} \right) + \left( \tilde{v}^r \partial_r \tilde{v}^r
 + \tilde{v}^z \partial_z \tilde{v}^z - \frac{(\tilde{v}^\theta )^2}{r} \right),
\ee
\be \nonumber
\tilde{F}^z = - (v^r \partial_r v^z + v^z \partial_z v^z ) + (\tilde{v}^r \partial_r \tilde{v}^z + \tilde{v}^z \partial_z \tilde{v}^z ) ,
\ee
and
\be \nonumber
\tilde{F}^\theta = -\left(v^r \partial_r v^\theta + v^z \partial_z v^\theta + \frac{v^r v^\theta}{r} \right) + \left( \tilde{v}^r \partial_r \tilde{v}^\theta + \tilde{v}^z \partial_z \tilde{v}^\theta + \frac{\tilde{v}^r \tilde{v}^\theta }{r }  \right)
\ee
Furthermore, one has
\be \nonumber
\|\tilde{F}^r \|_{L^2(\Omega)} \leq C \Phi^{\frac{7}{80} } \| \Bv^r - \tilde{\Bv}^r \|_{H^{\frac54}(\Omega) } +
C \Phi^{-\frac{3}{40} } \|\Bv^z - \tilde{\Bv}^z \|_{H^{\frac54}(\Omega)} + C \Phi^{\frac{1}{40}} \|\Bv^\theta - \tilde{\Bv}^\theta \|_{H^{\frac54}(\Omega)},
\ee
\be \nonumber
\|\tilde{F}^z \|_{L^2(\Omega)} \leq C \Phi^{\frac{7}{80}} ( \|\Bv^r - \tilde{\Bv}^r \|_{H^{\frac54}(\Omega) } + \|\partial_z \Bv^z -\partial_z \tilde{\Bv}^z \|_{H^{\frac14}(\Omega)})  + C \Phi^{-\frac{3}{40}} \|\Bv^z - \tilde{\Bv}^z \|_{H^{\frac54}(\Omega)},
\ee
and
\be \nonumber \ba
\|\tilde{F}^\theta \|_{L^2(\Omega)} \leq & C \Phi^{\frac{1}{40}} \|\Bv^r -\tilde{\Bv}^r \|_{H^{\frac54}(\Omega)} + C \Phi^{-\frac{3}{40}} \|\Bv^\theta
- \tilde{\Bv}^\theta \|_{H^{\frac54}(\Omega)} \\
&\ \ \  + C \Phi^{-\frac{9}{40}} \|\Bv^z - \tilde{\Bv}^z \|_{H^{\frac54}(\Omega)} + C \Phi^{\frac{7}{80} }
\|\partial_z \Bv^\theta - \partial_z \tilde{\Bv}^\theta \|_{H^{\frac14} (\Omega) }.
\ea
\ee
Hence by virtue of Lemma \ref{linearized-estimates-new}, it holds that
\be \nonumber
\ba
& \Phi^{\frac15} \|\Bv^r - \tilde{\Bv}^r \|_{H^{\frac54}(\Omega)} + \|\Bv^z - \tilde{\Bv}^z \|_{H^{\frac54} (\Omega)} +
\Phi^{\frac15} \|\partial_z \Bv^z - \partial_z \tilde{\Bv}^z \|_{H^{\frac14} (\Omega) } \\
&\ \ \ \ \ \ \ + \Phi^{\frac{1}{10}} \|\Bv^\theta - \tilde{\Bv}^\theta \|_{H^{\frac54}(\Omega)} + \Phi^{\frac{1}{5}}\|\partial_z \Bv^\theta - \partial_z \tilde{\Bv}^\theta \|_{H^{\frac14}(\Omega)} \\
\leq & C  \left( C_1 \Phi^{\frac15} \Phi^{-\frac{23}{160}} + C_1 \Phi^{\frac{1}{16}} + C_1 \Phi^{\frac15 - \frac{29}{112}}   \right) \left( \Phi^{\frac{7}{80} - \frac15 } \Phi^{\frac15} \|\Bv^r - \tilde{\Bv}^r \|_{ H^{\frac54}(\Omega)} \right.\\
&\  \left. +  \Phi^{-\frac{3}{40}} \|\Bv^z - \tilde{\Bv}^z \|_{H^{\frac54}(\Omega)} +  \Phi^{\frac{7}{80}- \frac15 }\Phi^{\frac15}
\|\partial_z \Bv^z - \partial_z \tilde{\Bv}^z \|_{H^{\frac14}(\Omega)} +  \Phi^{\frac{1}{40}- \frac{1}{10}} \Phi^{\frac{1}{10}} \|\Bv^\theta -
\tilde{\Bv}^\theta \|_{H^{\frac54}(\Omega)}\right) \\
& + C \left(C_2 \Phi^{\frac{1}{10}} +C_2 \Phi^{-\frac38 + \frac15}\right) \left( \Phi^{\frac{1}{40}- \frac15} \Phi^{\frac15} \|\Bv^r - \tilde{\Bv}^r \|_{H^{\frac54}(\Omega)} + \Phi^{-\frac{3}{40} - \frac{1}{10} } \Phi^{\frac{1}{10}}  \|\Bv^\theta - \tilde{\Bv}^\theta \|_{H^{\frac54}(\Omega)}        \right.  \\
&  \ \ \ \ \left.+ \Phi^{-\frac{9}{40}} \|\Bv^z - \tilde{\Bv}^z \|_{H^{\frac54}(\Omega)} + \Phi^{\frac{7}{80} - \frac15} \Phi^{\frac15} \|\partial_z \Bv^\theta - \partial_z \tilde{\Bv}^\theta \|_{H^{\frac14}(\Omega)} \right)\\
\leq & C \Phi^{-\frac{1}{20}} \Phi^{\frac15} \|\Bv^r - \tilde{\Bv}^r \|_{ H^{\frac54}(\Omega)} +
C \Phi^{-\frac{1}{80}} \|\Bv^z - \tilde{\Bv}^z \|_{H^{\frac54}(\Omega)} + C \Phi^{-\frac{1}{20}} \Phi^{\frac15} \|\partial_z \Bv^z - \partial_z \tilde{\Bv}^z \|_{H^{\frac14} (\Omega) }\\
&\ \ \ \ \ \ \ + C \Phi^{-\frac{1}{80}} \Phi^{\frac{1}{10}} \|\Bv^\theta - \tilde{\Bv}^\theta \|_{H^{\frac54}(\Omega)} + C \Phi^{-\frac{1}{80}}\Phi^{\frac{1}{5}}\|\partial_z \Bv^\theta - \partial_z \tilde{\Bv}^\theta \|_{H^{\frac14}(\Omega)}.
\ea
\ee
When $\Phi \geq \Phi_2$ is large enough, one has $\Bv = \tilde{\Bv}$. Therefore, the uniqueness is proved.

 The proof for part (a) of Theorem \ref{mainthm2} is almost the same as that of Theorem \ref{mainthm}, so we omit the details.
\end{proof}

%%%%%%%%%%%%%%%%%%%%%%%%%%Appendix%%%%%%%%%%%%%%%%%%%%%%%%%%%%%%%%%%%%%%%%%%%%%%

%\newpage

\appendix
\section{Some elementary lemmas}
In this appendix, some elementary lemmas are collected together.  They play important roles in the paper and might be useful elsewhere. The proofs of these lemmas can be referred to  \cite{WX1} and are omitted here.

The first one is the following Poincar\'e type inequalities.
\begin{lemma}\label{lemma1}
For a  function $g\in C^2([0,1])$
it holds that
\be \label{2-1-11}
\int_0^1 |g |^2  r \, dr \leq  \int_0^1 \left|   \frac{d}{dr} (r g )     \right|^2 \frac1r \, dr.
\ee
If, in addition, $g(0)=g(1)=0$, then one has
\be
  \int_0^1 \left|   \frac{d}{dr} (r g )     \right|^2 \frac1r \, dr
  \leq \left( \int_0^1 | \mathcal{L}g|^2  r \, dr  \right)^{\frac12} \left( \int_0^1 |g|^2 r\, dr  \right)^{\frac12}
\leq  \int_0^1 | \mathcal{L}g|^2  r \, dr.
\ee
\end{lemma}

%%%%%%%%%%%%%%%%%%%%%%%%%%%%%%%%%%Hardy-Littlewood-Polya%%%%%%%%%%%%%%5
The following lemma is a variant of Hardy-Littlewood-P\'olya type inequality.
\begin{lemma}\label{lemmaHLP}
Let $g\in C^1([0,1])$ satisfy $g(0)=0$, one has
\begin{equation}\label{ineqHLP}
\int_0^1|g(r)|^2 dr \leq \frac{1}{2} \int_0^1 |g^{\prime}(r)|^2 (1-r^2) \, dr,
\end{equation}
and
\be \label{HLP-2}
\int_0^1 |g|^2 r \, dr \leq C \int_0^1 \left| \frac{d(r g) }{dr}   \right|^2 \frac{1-r^2}{r} \, dr .
\ee
\end{lemma}
The following lemma is about two weighted interpolation inequalities, which are quite similar to \cite[(3.28)]{M}.
\begin{lemma}\label{weightinequality} Let $g \in C^2[0, 1]$, then one has
\begin{equation} \label{weight1} \ba
\int_0^1 |g|^2r \, dr  \leq & C \left(\int_0^1 (1-r^2)|g|^2 r \, dr\right)^{\frac23} \left(\int_0^1  \left|\frac{d}{dr}(rg)\right|^2 \frac{1}{r} \, dr\right)^{\frac13} \\
&\ \ \ \  + C \int_0^1 (1-r^2)|g|^2 r\, dr
\ea
\end{equation}
and
\be \label{weight2} \ba
\int_0^1 \left| \frac{d}{dr}(rg) \right|^2 \frac1r \, dr &  \leq C \left( \int_0^1 \frac{1-r^2}{r} \left| \frac{d}{dr} ( rg)  \right|^2 \, dr \right)^{\frac23} \left( \int_0^1 |\mathcal{L} g|^2 r \, dr  \right)^{\frac13} \\
&\ \ \ \ \ \ \ \ \ \ \ + C \int_0^1 \frac{1-r^2}{r} \left| \frac{d}{dr} ( rg)  \right|^2 \, dr .
\ea
\ee
\end{lemma}

%%%%%%%%%%%%%%%%%%%%%%%Sobolev%%%%%%%%%%%%%%%%%%%%%%%%%%%%%%%%%5555
\begin{lemma}\label{sobolev}
Assume that $\alpha > 0$. There exists a constant $C$, such that for every $g\in H^1(B_1^4(0))$, one has
\be \nonumber
\int_{B_1^4(0)} |g|^2 \, dx \leq C \left[ \int_{B_1^4(0) } |\nabla g|^2 \, dx + \alpha \int_{\partial B_1^4(0)} |g|^2 \, dS   \right].
\ee
\end{lemma}

%%%%%%%%%%%%%%%%%%%%%Bessel%%%%%%%%%%%%%%%%%%%%%%%%%%55
%Some basic properties of the modified Bessel functions of the first kind are collected in the following lemma.
%\begin{lemma}\label{lemBessel}
%Let $I_1(z)$ be the modified Bessel function of the first kind, i.e. it is the solution of the problem
%\eqref{eqBessel1}.
%Assume that $0 < x <y$, it holds that
%\be \label{Bessel1}
%e^{x - y} \frac{x}{y} < \frac{I_1(x)}{I_1 (y)} < e^{x - y} \left(\frac{y}{x}\right)^{1/2}.
%\ee
%Furthermore, for every $x>0$, one has
%\be\label{Bessel1-5}
%\frac{x}{2}\leq  I_1(x)  \leq \frac{x}{2} \cosh x
%\ee
%and
%\be \label{Bessel2}
% 0\leq I_1^{\prime}(x)  \leq I_1(x)+\frac{I_1(x)}{x}.
%\ee
%\end{lemma}

%\begin{lemma}\label{AlemBessel2}
%It holds that
%\be \label{A-96}
%\int_0^1 | I_1(|\xi| r)|^2 r\, dr
%\leq  C \min\{1, |\xi|^{-1}\} (I_1(|\xi|))^2
%\ee
%and
%\be \label{A-97}
% \int_0^1 \left| \frac{d}{dr} \left(r I_1 (|\xi| r) \right) \right|^2 \frac1r \, dr
%\leq  C \max\{1, |\xi|\} (I_1(|\xi|))^2.
%\ee
%Furthermore, one has
%\be\label{A-98}
%\int_0^1 | \mathcal{L}I_1(|\xi| r)|^2 r\, dr
%\leq  C  \min\{1, |\xi|^{-1}\} \xi^4 (I_1(|\xi|))^2
%\ee
%and
%\be\label{A-99}
%\int_0^1 \left| \frac{d}{dr} \left( r \mathcal{L} I_1(|\xi|r)\right)  \right|^2 \frac1r \, dr
%\leq  C  \max\{1, |\xi|\} \xi^4 (I_1(|\xi|))^2.
%\ee
%\end{lemma}

The following fixed point theorem is the basic tool to prove the nonlinear structural stability.
\bl \la{nonlinear} Let $Y$ be a Banach space with the norm $\|\cdot\|_{Y}$, and let $\mathcal{B}:\  Y\times Y\rightarrow Y$ be a bilinear map. There exists a constant $\eta>0$ such that  for all $\zeta_1, \zeta_2 \in Y$, one has
\be \nonumber
\|\mathcal{B}(\zeta_1, \zeta_2) \|_{Y} \leq \eta \|\zeta_1\|_{Y}  \|\zeta_2\|_{Y}.
\ee
For each $\zeta^* \in Y$ satisfying $4 \eta \| \zeta^* \|_{Y} < 1$, the equation
\be \nonumber
\zeta = \zeta^*  + \mathcal{B}(\zeta, \zeta)
\ee
has a unique solution $\zeta \in Y$ satisfying
\be \nonumber
\|\zeta\|_{Y} \leq 2\|\zeta^*\|_{Y}.
\ee
\el

%%%%%%%%%%%%%%%%%%%%%%%%%%%%%%%%%%%%%%%%%%%%%%%%%%%%%%%%%%%%%%%%%%%%

{\bf Acknowledgement.}
This work is financially supported by the National Key R\&D Program of China, Project Number 2020YFA0712000. The research of Wang was partially supported by NSFC grants 12171349 and 11671289. The research of Xie was partially supported by NSFC grants 11971307 and 11631008,  Natural Science Foundation of Shanghai 21ZR1433300. The authors would like to thank Professors Congming Li and Zhouping Xin for helpful discussions.

\end{document}